\documentclass[11pt,a4paper]{article}
\usepackage[latin1]{inputenc}
\usepackage{amsmath}
\usepackage{amsfonts}
\usepackage{amssymb}
\usepackage{graphicx}
\usepackage{amsthm}
\usepackage[width=17.10cm, height=25.00cm]{geometry}
\usepackage{color}
\usepackage{pict2e}
\usepackage{epic}

\newtheorem{theorem}{Theorem}[section]
\newtheorem{lemma}[theorem]{Lemma}

\newtheorem{remark}[theorem]{Remark}

\newcommand{\norm}[2]{\|{#1}\|_{{#2}}}
\newcommand{\tnorm}[1]{\|\kern-.4mm| {#1} |\kern-.4mm\|}
\newcommand{\ud}{\,{\rm d}}
\newcommand{\jump}[1]{[\kern-0.6mm [{#1}]\kern-0.6mm]}
\newcommand{\av}[1]{\{\kern-1.5mm \{{#1}\}\kern-1.5mm\}}
\newcommand{\ex}{{\rm e}}
\newcommand{\nx}{{\nabla}_{\! x}}
\newcommand{\nv}{{\nabla}_{\!v}}
\newcommand{\nvx}{{\nabla}_{\! v,x}}

\newcommand{\tltwo}[2]{(\kern-0.6mm ({#1},{#2})\kern-0.6mm)}

\newcommand{\tmesh}{\tau_{n}}

\newcommand{\stnfes}[1]{V^{q_{#1}}_h}

\newcommand{\stfes}{\mathcal{V}^\textbf{q}_{h}}

\newcommand{\ujump}[1]{\lfloor #1\rfloor}

\newcommand{\mbf}[1]{\mathbf{{#1}}}

\newcommand{\dip}[2]{(\!({#1},{#2})\!)}

\newcommand{\Anorm}[1]{|\!|\!| {#1} |\!|\!|}

\newcommand{\bx}{\mbf{x}}
\newcommand{\by}{\mbf{y}}

\newcommand{\bs}{\mbf{s}}

\newcommand{\dlangle}{\langle\kern-0.6mm \langle}
\newcommand{\drangle}{\rangle\kern-0.6mm \rangle}

\author{Zhaonan Dong\thanks{
		1) Inria, 48 rue Barrault, 75647 Paris, France
		and 2) CERMICS, ENPC, Institut Polytechnique de Paris, CNRS, 6 \& 8 avenue B.~Pascal, 77455 Marne-la-Vall\'{e}e, France.   Email: {\tt zhaonan.dong@inria.fr}}
\and Emmanuil H.~Georgoulis\thanks{
	The Maxwell Institute for Mathematical Sciences and Department of Mathematics, Heriot-Watt University, Edinburgh
EH14 4AS, United Kingdom,
	{AND} Department of Mathematics, School of Applied Mathematical and Physical Sciences, National Technical University of Athens, Zografou 15780, Greece,
{AND} Institute of Applied and Computational Mathematics, FORTH, 700 13, Heraklion, Crete, Greece. Email: {\tt E.Georgoulis@hw.ac.uk}
}}

\title{Hypocoercivity-preserving space-time Galerkin methods \\ for kinetic Fokker-Planck equations}
\begin{document}
	\maketitle

	\begin{abstract}
We design and analyse a family of hypocoercivity-preserving fully discrete Galerkin methods for the (inhomogeneous) kinetic Fokker--Planck (kFP) equations, a class of evolution partial differential equations (PDE) with \emph{degenerate} diffusion. The key idea in the method design is to mimic the framework of enhanced quadratic forms of Villani \cite{villani} and, thus, to arrive at a coercive bilinear form in a sufficiently strong exponentially weighted norm, admitting spectral gap/Poincar\'e inequality, despite the degeneracy. The problem is posed as a fourth order  evolution PDE on the whole space $\mathbb{R}^{d}\times\mathbb{R}^d$. The spatial discretization employs continuous piecewise polynomial approximation spaces, subordinate to simplicial and/or box-type meshes comprising finite, as well as `infinite' elements. The non-conformity is treated by constructing appropriate numerical fluxes in the spirit of $C^0$-interior penalty ($C^0$-IP) methods for fourth order elliptic problems.  The proof of coercivity of the method requires new trace inverse inequalities for exponentially weighted norms on $d$-dimensional simplicial and box-type domains, and also for prismatic domains with infinite length. We prove weighted trace inverse inequalities in all cases and for a variety of exponential weights. Once coercivity is established for the Galerkin method, the proof of exponential convergence to the equilibrium state follows by employing an exponentially weighted Poincar\'e inequality. A fully discrete extension of the method is also considered
by coupling the non-standard spatial discretization with an $hp$-version discontinuous Galerkin time-stepping scheme. The exponential convergence to the equilibrium state for the fully discrete scheme is also proven. The proposed methods preserve the total mass and exhibit \emph{provably} exponential convergence to the equilibrium state, a manifestation of `numerical hypocoercivity,' making them well suited for long-time kFP simulations.  A series of numerical experiments is presented to validate the theoretical results and investigate the convergence behaviour of the proposed method.
\end{abstract}

\section{Introduction}
	 Numerous physical, chemical, biological, and social dynamic processes are often characterised by convergence to long-time equilibria. In many important cases the diffusion/dissipation required to arrive to such equilibria is explicitly present in \emph{some} of the spatial directions only. This, somewhat counter-intuitive at first, state of affairs suggests that decay to equilibrium is due to finer hidden structure, which allows for the transport terms in the dynamics to `propagate dissipation' also to the directions which no explicit dissipation exists in the model.

An archetypical  model problem is the classical inhomogeneous kinetic Fok\-ker-Planck equation, reading: find  $f:(0,t_f] \times\mathbb{R}^d\times \mathbb{R}^d\to \mathbb{R}$, $d\in\mathbb{N}$, such that
	\begin{equation}\label{eq:FP}
 f_t+v\cdot \nx f-\nabla V(x)\cdot \nv f-\nv \cdot(\nv  f+vf)= 0,\quad \text{ for }(t,v,x)\in(0,t_f]\times \mathbb{R}^d\times \mathbb{R}^d,
	\end{equation}
	with $\nx$, $\nv $ denoting the gradient with respect to the components of $x$ and $v$, respectively, for some potential $V:\mathbb{R}^d\to\mathbb{R}$, 
	subject to the initial condition
	\begin{equation}\label{eq:ic}
	f(0,v,x)=f_0(v,x),\qquad (v,x)\in \mathbb{R}^d\times \mathbb{R}^d.
	\end{equation}
	The solution $f$ is understood as (a multiple of a) probability density function for a particle with velocity $v$ to be in position $x$ at time $t$.
 The external potential $V$ is assumed to satisfy $V\in C^2(\mathbb{R}^d)$ and
	\begin{equation}\label{eq:Vgrowth}
	|\mathcal{H}(V)|_F\le C_0 (1+|\nabla V|),
	\end{equation}
	pointwise for all $x\in\mathbb{R}^d$ for some $C_0>0$, with  $\mathcal{H}(V)$ denoting the Hessian of $V$ and $|\cdot|_F$ the standard Frobenius matrix norm. The potential $V$ confines the particles from escaping to infinity. For the numerical analysis below, we will further quantify $V$ to be of the form $V(x)=\phi(x) + |x|^{2k}$, $k\in\mathbb{N}$, for $\phi$ bounded function.

	Assuming that $f_0\in L_2((1+E)\ud v\ud x)$, where
	$
	L_2((1+E)\ud v\ud x):=\{f: \int |f|^2  (1+E)\ud v\ud x<+\infty\}$,
	and
	\begin{equation}\label{exponent}
		E(v,x):= V(x)+\frac{1}{2}|v|^2,
	\end{equation} the problem \eqref{eq:FP}, \eqref{eq:ic} has a unique distributional solution \cite[Theorem 7]{villani}; see also \cite{helffer_nier} for the first result in this direction under stronger regularity and growth assumptions on $V$. To avoid notational overload, we do \emph{not} normalise $ \ex^{-E}$ and $f$.

To highlight the challenges arising from the explicit presence of diffusion in the $v$-variables \emph{only}, we work as follows.	Performing the standard substitution $f= u \ex^{-E}$ into \eqref{eq:FP}, we deduce
		\begin{equation}\label{eq:FP_u_classical}
	u_t+v\cdot \nx  u-\nabla V\cdot \nv u=\Delta_v u -v \cdot \nv u, \qquad u(0,\cdot,\cdot)=u_0(\cdot,\cdot),
	\end{equation}
	with $u_0:=f_0\ex^E\in L_2(\mu)$, with $\ud \mu := \ex^{-E(v,x)}\ud v\ud x$; see Section \ref{sec:weak_form} for the formal definitions of the weighted inner product and function spaces. Under these assumptions, \eqref{eq:FP_u_classical} admits a unique distributional solution; see \cite[Theorem 6]{villani} for details.

  In spite of the degeneracy of the diffusion in \eqref{eq:FP_u_classical}, the solution $u$ decays to the equilibrium state $\int u_0\ud \mu$ exponentially fast in time. This is a manifestation of the concept of \emph{hypocoercivity} introduced by Villani in \cite{villani}. Still, employing a straightforward energy argument does not allow for proof of decay to equilibrium, as no evident spectral gap property is apparent. Indeed, observing the identity $\Delta_v u -v \cdot \nv u =\nv \cdot(\nv  u \ex^{-E})\ex^E$, \eqref{eq:FP_u_classical} gives
	\begin{equation}\label{eq:FP_u_exp}
  	u_t+v\cdot \nx u-\nabla V\cdot \nv u-\nv \cdot(\nv  u \ex^{-E})\ex^E= 0,
  \end{equation}
  which, upon testing against $u\ex^{-E}$ and integrating over $\mathbb{R}^d\times \mathbb{R}^d$, gives
  \begin{equation}\label{eq:fp_energy}
\frac{1}{2}\frac{\ud}{\ud t}\norm{u}{L_2(\mu)}^2+\norm{\nv  u}{L_2(\mu)}^2=0,
  \end{equation}
  since $(v\cdot \nx  u-\nabla V \cdot \nv  u,u)_{L_2(\mu)}=0$, for $\ud \mu = \ex^{-E(v,x)}\ud v\ud x$.

  At this point it is useful to note the conservation of mass property of \eqref{eq:FP_u_classical}:
  \begin{equation}\label{eq:fp_cons_mass}
  	\int u\ud \mu= 	\int u_0\ud \mu=:\bar{u},\quad t\in[0,t_f],
  \end{equation}
  for all $t_f>0$, which follows by testing \eqref{eq:FP_u_classical} against $\ex^{-E}$, (i.e., test against the constant function $1$ in the weighted inner product of $L_2(\mu)$,)  and integrating over $\mathbb{R}^d\times\mathbb{R}^d$.  Thus, one would need a spectral gap/Poincar\'e-type inequality of the form
  \begin{equation}\label{eq:PF_gen}
\ \norm{u-\int u\ud \mu }{L_2(\mu)}^2\le \frac{2}{\kappa}\norm{\nv  u}{L_2(\mu)}^2,
  \end{equation}
  for some $\kappa>0$,
  to yield decay to equilibrium. Then, substituting \eqref{eq:PF_gen} into \eqref{eq:fp_energy} and employing Gr\"onwall Lemma, we would arrive at
\begin{equation}\label{hypco_decay}
 \norm{u(t_f)-\bar{u}}{L_2(\mu)}^2\le \ex^{-\kappa t_f} \norm{u_0-\bar{u}}{L_2(\mu)}^2,
  \end{equation}
  i.e., exponential decay to the equilibrium state $\bar{u}$ as $t_f\to\infty$.  However, \eqref{eq:PF_gen} is \emph{not} true, in general, since $\nv  u$ vanishes for all functions $u(t,v,x)=w(x)$. The crucial challenge, therefore, is to circumvent the lack of coercivity of the natural energy quadratic form.

  The proof of trend to equilibrium for the kinetic Fokker-Planck equation was first given by H\'erau \& Nier in  \cite{herau_nier} upon making the crucial observation that certain entropies admitting mixed derivatives $f_{vx}$ give rise to full gradients in certain $\mu$-weighted Sobolev spaces. Some related ideas  can also be found in the work of Eckmann \& Hairer  \cite{eckmann_hairer}, whereby the spectral properties of certain hypoelliptic operators involving degenerate diffusions are studied, and by Mouhot \& Neumann \cite{mouhot_neumann} in the context of kinetic models with integral-type collision operators, among others. The key observation of using quadratic functionals involving mixed derivatives was elevated to a general framework in proving decay to equilibrium for kinetic equations by Villani \cite{villani} via the introduction of the concept of \emph{hypocoercivity}. Roughly speaking, hypocoercivity is the property of certain degenerate evolution operators to yield dissipation also in the directions where \emph{no} diffusion is explicitly present. 
  By employing non-standard quadratic forms, Villani proved coercivity in sufficiently strong norms, for which Poincar\'e inequalities apply.   The theory of hypocoercivity has been applied to various settings in recent years; see \cite{villani} for a non-exhaustive sample. Further approaches in the context of hypocoercivity have appeared since, starting with the seminal work of Dolbeault, Mouhot \& Schmeiser \cite{DMS_TAMS}.

  The importance of designing computational methods for kinetic equations cannot be overstated, due to their widespread use of in modern statistical physics, which has sprawled into other modelling areas  in recent years \cite{Risken,pareschi_toscani}.  Monte--Carlo (MC) approaches have enjoyed the attention of the computational physics and chemistry community over the last 50 years or so \cite{pareschi_toscani}. The slow and statistical-in-nature convergence rate of MC simulations and their structural difficulty in retaining conservation properties of the solutions at the discrete level, have lead researchers to seek alternatives; see \cite{dimarco_pareschi} and the references therein, including finite differences, spectral methods, and Galerkin/finite element approaches. In particular, discontinuous Galerkin (dG) methods have been used for the approximation of Fokker-Planck; see, the influential work \cite{newpaper} for a dG method for a class of equations including Fokker-Planck.

  Despite their important scientific value, the aforementioned works are not concerned with preserving hypocoercivity-inducing structures at the discrete level.  Nonetheless, preserving such hypocoercivity properties appears to be important to guarantee quantitatively correct simulations, especially for long-time simulations. To see this, we write \eqref{eq:FP} in the abbreviated form $f_t+\mathcal{L}f=0$ and we consider a discretization method
  \[
   \partial f_h +\mathcal{L}_hf_h = 0,
   \] of \eqref{eq:FP}. Let $e:=f-f_h$ be the discretization error. Subtracting the method from the equation gives
   \[
    \partial e +\mathcal{L}_he = \mathcal{R}
    \] with $\mathcal{R}:=\partial f-f_t+(\mathcal{L}_h-\mathcal{L})f $ the discretization residual.
  If $\mathcal{L}_h$ does not preserve hypocoercivity, we cannot infer a spectral gap for $\mathcal{L}_h$. Thus, we can only estimate the error via a classical energy argument, for which classical (discrete) Gr\"onwall-type estimates give bounds of the form
  \begin{equation}\label{gronwall}
  \norm{e(t)}{L_2(\mu)}^2\le e^{\delta t} \big(\delta^{-1}\norm{\mathcal{R}}{L_2(I;L_2(\mu))}^2+\norm{e(0)}{L_2(\mu)}^2\big).
  \end{equation}
  for any $\delta>0$ and $I=(0,t)$; such estimates hold for a number of classical choices of temporal discretization $\partial$. The same argument also holds for the case of \eqref{eq:FP} with non-homogeneous right-hand sides, which can be important in applications \cite{pareschi_toscani}. Hence, not preserving the hypocoercivity structure of $\mathcal{L}$ at the discrete level, may lead to exponentially diminishing performance or, at least, degraded accuracy, with respect to the final time. Given the importance of kinetic simulations in various high-risk (e.g., nuclear, space, etc.) industries, the development of numerical methods with realistic error behaviour is of significance.

  To address this challenge, it is desirable to port hypocoercivity ideas to numerical methods. To date, there exist extremely few, yet quite inspiring, works investigating the hypocoercivity properties of finite difference operators.  By applying the abstract semigroup result of Villani \cite[Theorem 18]{villani}, Porretta \& Zuazua \cite{porretta_zuazua} prove the hypocoercivity of classical finite difference operators discretizing the related Kolmogorov's equation. This was later extended by Dujardin, Herau, \& Lafitte \cite{dujardin_herau_lafitte} to proof of decay to equilibrium for finite difference approximations for the Fokker-Planck equation \eqref{eq:FP}, \eqref{eq:ic} for the case $d=1$. Further, hypocoercivity structures for finite volume discretizations of various kinetic equations for $d=1$ are presented in \cite{BessemoulinHerdaRey2024,BessemoulinLadineRey2025,BlausteinFilbet24}. We also note the related works \cite{foster_loheac_tran} and \cite{besse_filbet} which, nonetheless, do not focus on proof of hypocoercivity at the fully discrete level.

  The first proof of a non-standard hypocoercivity-compatible Galerkin finite element method for the Kolmogorov equation was given \cite{georgoulis21}. The key idea in \cite{georgoulis21} is the variational interpretation of the hypocoercivity structure as a degenerate 4th-order PDE problem in weak form. The respective numerical method, based on this weak form, employs $C^0$-Lagrange finite elements of any order in space.
  We also mention the recent result \cite{SUPGhypo25}, whereby a carefully constructed streamline-upwinded Galerkin method is shown to admit a mesh-dependent numerical hypocoercivity property for the Kolmogorov equation, followed by a corresponding result for the classical space-time dG method for the Kolmogorov problem in \cite{dong:hal-05080993}.

In this work, we are concerned with designing and analysing the first provably hypocoercivity-preserving fully discrete Galerkin method for the kinetic Fokker-Planck (kFP) equation \eqref{eq:FP}, \eqref{eq:ic}. This is achieved by first constructing a special variational form of of the kFP equation in the spirit of \cite{georgoulis21} and, subsequently, by addressing  the additional challenges posed by 1) the lack of coercivity  of the resulting bilinear form due to the presence of the confining potential $V$,  2) by the potential lack of $C^1$-regularity of the Galerkin approximations for the 3rd and 4th order partial derivatives appearing in the special variational form, and 3) the problem posed on the whole of $\mathbb{R}^d\times \mathbb{R}^d$. To address the first challenge, we extend a technical result by Villani from \cite{villani} estimating the Hessian of the potential $V$ (see Lemma \ref{lem:villani_control} below).

The second challenge is considerably more involved and requires a number of technical developments. First, the 3rd order derivatives are treated via a non-standard set of numerical fluxes, which are able to transport information to the correct directions without introducing additional numerical diffusion and, crucially, are also compatible with respect to the algebraic structure that we seek to preserve. Furthermore, to address the discretization of the 4th-order partial derivatives in $v$ appearing in the modified variational form, we employ a non-conforming $C^0$-interior penalty-type set of numerical fluxes. The proof of stability of such discretizations require the availability of, so-called, trace inverse inequalities bounding norms of polynomials on the boundary of a domain (element) by the same norms on the domain itself. To that end, we prove new trace inverse inequalities in $L_2$-norms with a variety of exponentially weighted measures. Moreover, we also prove trace inverse inequalities with exponential weights on \emph{unbounded} domains also. The proof of these estimates hinges on new, to the best of our knowledge, inverse estimates in one dimension. The above may be of independent interest.

The third challenge, i.e., the problem posed on the whole $\mathbb{R}^d\times \mathbb{R}^d$ for the spatial variables is addressed by considering piecewise polynomials spaces subordinate to subdivisions (meshes) of $\mathbb{R}^d\times \mathbb{R}^d$ comprising both \emph{bounded as well as unbounded} subdomains; we refer to Figure \ref{fig:meshes} for an illustration. In particular, in regions of significant evolution, typical finite element meshes are used, which are further extended to the whole space by `infinite' elements.  Once coercivity is established,  for the Galerkin method, the proof of exponential convergence to the equilibrium state follows by employing an exponentially weighted Poincar\'e inequality.

The spatial discretization is complemented by an $hp$-version discontinuous Galer\-kin time-stepping method, allowing for arbitrary order temporal approximations and facilitating the design of local time-stepping. Also, we discuss briefly the challenge of high-dimensionality when $d\ge2$.

The remainder of this work is structured as follows. In Section~\ref{sec:weak_form}, we introduce a nonstandard weak formulation of the model problem, before proving the hypocoercivity property for the PDE in Section~\ref{sec: hypocoecivity}. In Section~\ref{sec: Numerical methods}, we introduce a $C^0$-interior penalty Galerkin method for the spatial discretization. Section~\ref{sec: Inverse estimates with exponential weights} is devoted to the derivation of novel inverse trace inequalities with exponential weights for bounded and unbounded domains. These inequalities are then employed in Section~\ref{sec: Numerical hypocoercivity} to establish numerical hypocoercivity for the spatially semidiscrete scheme. In Section~\ref{sec: A fully discrete scheme}, we introduce a discontinuous Galerkin time-stepping method and establish numerical hypocoercivity for the fully discrete scheme. In Section~\ref{sec: Numerical examples} we provide some numerical experiments showcasing the practical relevance of the proposed method. Finally, in Section~\ref{sec: conclusion}, we draw some brief conclusions.

	\section{A non-standard weak form}\label{sec:weak_form}
	In the spirit of the seminal work of H\'erau \& Nier \cite{herau_nier}, which led to the celebrated general theory of hypocoercivity by Villani \cite{villani}, we will construct a weak form for \eqref{eq:FP_u_classical}, giving rise to a quadratic form involving norms of $\nx$ and of the mixed derivatives $\nv \nx^T u$.

Consider the equilibrium measure
	$
	\ud \mu = \ex^{-E(v,x)}\ud v\ud x
	$.
	We denote by $L_2(\mu)\equiv H^0(\mu)$ and, in general, $H^r(\mu)$, $r=0,1,2,\dots$ the weighted Hilbertian Sobolev spaces	\[
	H^r(\mu):=\Big\{w:\mathbb{R}^{d}\times \mathbb{R}^{d}\to\mathbb{R}:\norm{w}{H^{r}(\mu)}:=\Big(\int \big( \sum_{s=0}^r|\nvx^{\,s}w|^2\big)\ud \mu\Big)^{1/2}\Big\},
	\]
	with $\nvx :=(\nv^T ,\nx^T )^T$ and the integration taking place over $\mathbb{R}^d\times\mathbb{R}^d$; the respective inner product is denoted by $(\cdot,\cdot)_{H^r(\mu)}$ and defined by polarisation. For tensor-valued functions, we use the convention
	$
	\norm{\nv\nx^Tw}{L_2(\mu)}^2 = \sum_{i,j=1}^d\norm{w_{x_i,v_j}}{L_2(\mu)}^2,
	$
with analogous conventions for other tensor-valued quantities.

	In what follows, we prefer to work with \eqref{eq:FP_u_exp}, as it facilitates the explicit imposition of the decay of the solution at infinity.
	We begin by setting $\mathcal{V}:=(- \nabla V^T , v^T)^T$ for brevity.
	To that end, we first test \eqref{eq:FP_u_exp} against $w\ex^{-E}$ for $w\in H^1(\mu)$, to deduce
	\begin{equation}\label{eq:hypo_one}
		(u_t,w)_{L_2(\mu)}+(\mathcal{V}\cdot \nvx u,w)_{L_2(\mu)}+(\nv  u,\nv  w)_{L_2(\mu)}= 0.
	\end{equation}
	Next, we differentiate \eqref{eq:FP_u_classical} with respect to $\nvx$, which yields
\[
\nvx  u_t+ (\mathcal{V}\cdot\nvx ) \nvx  u
+
\Big(\begin{array}{cc}
I  & I \\
-\mathcal{H}(V) & 0
\end{array}\Big)\nvx  u
-\sum_{j=1}^d\nvx   u_{v_jv_j} +(v \cdot \nv )\nvx  u= 0,
\]
with $\mathcal{H}(V)$ denoting the Hessian matrix of $V$ with respect to $x:=(x_1,\dots,x_d)^T$; also $v:=(v_1,\dots,v_d)^T$ and $I\in\mathbb{R}^{d\times d}$ is the identity matrix. Testing now against $\mathcal{A}\nvx  w\ex^{-E}$, with $\mathcal{A}\in\mathbb{R}^{2d\times 2d}$ given by
	\[
	\mathcal{A}=\left(
	\begin{array}{cc}
	\alpha I & \beta I \\
	 \beta I & \gamma I
	 \end{array}
	\right),
	\]
	for $\alpha,\beta,\gamma>0$ to be determined below, we deduce
	\begin{equation}\label{eq:hypo_two}
	\begin{aligned}
	&0=(\nvx  u_t, \mathcal{A}\nvx  w)_{L_2(\mu)}+  ( (\mathcal{V}\cdot\nvx)\nvx  u, \mathcal{A}\nvx  w)_{L_2(\mu)}
\\
	+&
\left(	\left(
	\begin{array}{cc}
	\alpha I -\beta \mathcal{H}(V)& \alpha  I \\
	\beta I -\gamma \mathcal{H}(V) & \beta I
	\end{array}
	\right)\nvx  u,\nvx  w
	\right)_{L_2(\mu)}
+\sum_{j=1}^d(\nvx   u_{v_j} , \mathcal{A}\nvx  w_{v_j})_{L_2(\mu)},
	\end{aligned}
	\end{equation}
	upon observing the identity
	\[
	\begin{aligned}
	-\sum_{j=1}^d(\nvx   u_{v_jv_j} , \mathcal{A}\nvx  w)_{L_2(\mu)} \! =& \sum_{j=1}^d(\nvx   u_{v_j} , \mathcal{A}\nvx  w_{v_j})_{L_2(\mu)}\\
	&-((v \cdot \nv )\nvx  u, \mathcal{A}\nvx  w)_{L_2(\mu)}.
	\end{aligned}
	\]
	Adding \eqref{eq:hypo_one} to \eqref{eq:hypo_two} yields the non-standard weak form for \eqref{eq:FP_u_classical}: for a.e. $t\in(0,t_f]$, find $u\in H^2(\mu)$, such that
	\begin{equation}\label{eq:hypo_weak_form}
	(u_t, w)_{L_2(\mu)}+(\nvx  u_t, \mathcal{A}\nvx  w)_{L_2(\mu)}+a(u,w) = 0,
	\end{equation}
	for
	\[
	\begin{aligned}
	a(u,w):=&\  (\mathcal{V}\cdot \nvx  u,w)_{L_2(\mu)}+( (\mathcal{V}\cdot \nvx ) \nvx  u, \mathcal{A}\nvx  w)_{L_2(\mu)}
	\\
	&	+
	(		\mathcal{B}\nvx  u,\nvx  w)_{L_2(\mu)}
	+\sum_{j=1}^d(\nvx   u_{v_j} , \mathcal{A}\nvx  w_{v_j})_{L_2(\mu)}
	\end{aligned}
	\]
	with
	$
	\mathcal{B}\equiv \mathcal{B}(V,\mathcal{A}):=	\Big(
	\begin{array}{cc}
		(1+\alpha) I-\beta\mathcal{H}(V) & \alpha  I \\
		\beta I-\gamma\mathcal{H}(V) & \beta I
	\end{array}
	\Big).
	$

\section{Hypocoercivity}\label{sec: hypocoecivity}
Although \eqref{eq:hypo_weak_form} may appear cumbersome at first sight, it has the crucial property of being coercive with respect to a suitable norm for certain choices of the constants $\alpha,\beta$, and $\gamma$. This norm, in turn, is such that it admits a spectral gap/Poincar\'e inequality; a manifestation of hypocoercivity \cite{villani}.

Crucially, the term involving $\mathcal{H}(V)$ has no definite sign a priori in $a(\cdot,\cdot)$. Nonetheless, it is possible to show that, under the growth condition \eqref{eq:Vgrowth}, the following bound holds. This will allow to control the sign-indefinite terms involving $\mathcal{H}(V)$ and, eventually, conclude the proof of (hypo)coercivity.
\begin{lemma}[{\cite[Lemma A.24]{villani}}]\label{lem:villani_control}
Let $g\in H^1(\mu)$ and define the energy by $E(v,x)=V(x)+\frac{1}{2} |v|^2$, with $V\in C^2(\mathbb{R}^d)$ satisfying \eqref{eq:Vgrowth}. Then, for $c_1:=16C_0^2(1+\sqrt{2dC_0})^2$, we have the estimate
\begin{equation}\label{eq:villani_control}
\norm{\mathcal{H}(V) g}{L_2(\mu)}^2  \le c_1 \big(\norm{g}{L_2(\mu)}^2 +\norm{\nx g}{L_2(\mu)}^2\big).
\end{equation}
\end{lemma}
We are now ready to show the (hypo)coercivity of $a(\cdot,\cdot)$.

\begin{lemma}\label{lem:hypoco}
For  any $u\in H^2(\mu)$, there exist choices of $\alpha,\beta,\gamma>0$ with $\alpha\gamma-\beta^2\ge 0$ (i.e., $\mathcal{A}$ is non-negative definite,) so that we have
		\[
	\begin{aligned}
	a(u,u)\ge
&\ c_{hc,1}\big( \norm{\nv u}{L_2(\mu)}^2 +\beta \norm{\nx u}{L_2(\mu)}^2\big)\\
& +c_{hc,2}\alpha \norm{\nv\nv^T u}{L_2(\mu)}^2 + c_{hc,3}\gamma\norm{\nx\nv^T u}{L_2(\mu)}^2,
\end{aligned}
\]
for constants $c_{hc,i}>0$ $i=1,2$ depending only on $c_1$ for each choice of $\alpha,\beta$ and $\gamma$.
\end{lemma}
\begin{proof} A proof can be found in \cite{villani}, presented in an abstract and generic operator setting. We prefer to prove the result more directly for the current setting of the Fokker-Planck equation, to track the explicit dependence of the constants on the matrix $\mathcal{A}$ and, crucially, to highlight the interplay between differential operator commutations in the proof.

	We treat each terms in $a(\cdot,\cdot)$ separately. As before, since $\sqrt{\mathcal{A}}$ is well defined, we have
	$
	(\mathcal{V}\cdot \nvx  u,u)_{L_2(\mu)}=0$,
	and, also,
	\[
( (\mathcal{V}\cdot \nvx) \nvx  u, \mathcal{A}\nvx  u)_{L_2(\mu)} =	\frac{1}{2}\int (\mathcal{V}\cdot \nvx)|\sqrt{\mathcal{A}}\nvx  u|^2\ud \mu=0.
	\]

	Now, using Lemma \ref{lem:villani_control} with $g=\nabla_vu$, we deduce
	\begin{equation}\label{eq:HV_one}
	\beta (\mathcal{H}(V)\nv  u, \nv  u)_{L_2(\mu)}\le \frac{\beta^2\epsilon_1 c_1+\epsilon_1^{-1}}{2}\norm{\nv  u}{L_2(\mu)}^2 +\frac{\beta^2\epsilon_1 c_1}{2}\norm{\nx \nv^T  u}{L_2(\mu)}^2,
	\end{equation}
for any $\epsilon_1>0$. 
Similarly,
	\begin{equation}\label{eq:HV_two}
\gamma	(\mathcal{H}(V)\nv  u, \nx  u)_{L_2(\mu)} \le \frac{\gamma^2\epsilon_2 c_1}{2\beta}\big(\norm{\nv  u}{L_2(\mu)}^2+\norm{\nx \nv^T  u}{L_2(\mu)}^2\big) +\frac{\beta}{2\epsilon_2}\norm{\nx  u}{L_2(\mu)}^2 ,
	\end{equation}
	for any $\epsilon_2>0$.  Moreover, for any $\epsilon_3>0$, we have
\[
	\begin{aligned}
	\sum_{j=1}^d\!\norm{\sqrt{\mathcal{A}}\nvx   u_{v_j}}{L_2(\mu)}^2 \!=&\,
	 \alpha \norm{\nv\nv^T u}{L_2(\mu)}^2\!
	 +
	 \!2\beta (\nv \nv^T  u, \nx\nv^T   u)_{L_2(\mu)}\!+\! \gamma \norm{\nx\nv^T u}{L_2(\mu)}^2\\
	 \ge&\ \alpha (1-\epsilon_3) \norm{\nv\nv^T u}{L_2(\mu)}^2 + \Big(\gamma-\frac{\beta^2}{\alpha\epsilon_3}\Big)\norm{\nx\nv^T u}{L_2(\mu)}^2.
	 \end{aligned}
\]
	Finally, for any $\epsilon_4>0$, we have
	\[
	\begin{aligned}
		\left(	\Big(
	\begin{array}{cc}
		(1+\alpha) I & \alpha  I \\
		\beta I & \beta I
	\end{array}
	\Big)\nvx  u,\nvx  u\right)_{L_2(\mu)} 
	\ge&\  \Big( 1+\alpha - \frac{\epsilon_4(\alpha+\beta)^2}{2\beta }\Big) \norm{\nv u} {L_2(\mu)}^2\\
	&  + \beta\big(1-(2\epsilon_4)^{-1}\big) \norm{\nx u} {L_2(\mu)}^2.
	\end{aligned}
	\]
	Combining the above, we deduce
			\[
	\begin{aligned}
	a(u,u)\ge\
	&\Big(1+\alpha-\frac{\beta^2\epsilon_1 c_1+\epsilon_1^{-1}}{2}-\frac{\gamma^2\epsilon_2 c_1}{2\beta}-\frac{\epsilon_4(\alpha+\beta)^2}{2\beta}\Big)\norm{\nv u}{L_2(\mu)}^2 \\
	&\ +\beta\Big(1-(2\epsilon_2)^{-1}-(2\epsilon_4)^{-1}\Big)\norm{\nx u}{L_2(\mu)}^2+\alpha (1-\epsilon_3) \norm{\nv\nv^T u}{L_2(\mu)}^2\\
	&\  + \Big(\gamma-\frac{\beta^2}{\alpha\epsilon_3} -\frac{\beta^2\epsilon_1 c_1}{2}-\frac{\gamma^2\epsilon_2 c_1}{2\beta}\Big)\norm{\nx\nv^T u}{L_2(\mu)}^2 .
	\end{aligned}
	\]
We want to optimise now all the parameters $\epsilon_i$, $i=1,\dots,4$ and $\alpha,\beta,\gamma$ to maximise the coefficients of $\norm{\nv u}{L_2(\mu)}^2$ and $\norm{\nx u}{L_2(\mu)}^2$.  For instance, letting $\epsilon_1=(c_1\alpha)^{-1}$, $\epsilon_2=(2c_1\alpha)^{-1}$,  $\epsilon_3=5/6$, $\epsilon_4=3/5$,
as well as $\beta=\alpha^2/2$ and $\gamma=2\alpha^3/3$, this becomes
\[
\begin{aligned}
a(u,u)\ge\
&\Big(\frac{2}{5}+\frac{4-5c_1}{10}\alpha-\frac{3}{20}\alpha^2-\frac{25}{72}\alpha^3\Big)\norm{\nv u}{L_2(\mu)}^2 +\frac{\alpha^2}{2}\Big(\frac{1}{6}-\alpha c_1\Big)\norm{\nx u}{L_2(\mu)}^2\\
& +\frac{\alpha}{6}  \norm{\nv\nv^T u}{L_2(\mu)}^2 + \frac{7\alpha^3}{360}\norm{\nx\nv^T u}{L_2(\mu)}^2.
\end{aligned}
\]
Selecting now $\alpha>0$ small enough, the result follows.
\end{proof}

Thus, the presence of a full $H^1(\mu)$ seminorm allows for the use of a Poincar\'e inequality on weighted spaces, thereby, allowing to conclude exponential decay as $t_f\to\infty$.

\begin{remark}
Both the constant $c_{hc,1}$ and the admissible values of $\beta$ can be optimised via different choices of $\mathcal{A}$ to yield a larger (hypo)coercivity constant. The particular selection of $\mathcal{A}$ above is made in the interest of showing existence of at least one such matrix with the desired properties.
\end{remark}

\section{Numerical methods}\label{sec: Numerical methods}
We now proceed with the design of Galerkin methods for the approximation of the solution to \eqref{eq:FP_u_classical}. For the spatial discretization, we will employ continuous finite element spaces which, depending on the dimension $d$ of the problem, can be standard Lagrange finite elements (for $d=1$), or reduced complexity finite element spaces for $d\ge 2$, (such as sparse grids; see, e.g.,  \cite{griebel} for a survey,) or, perhaps,  a full $(2+2)$-dimensional finite element space in cases its complexity is computationally tractable. 
Nonetheless, the proofs below are \emph{oblivious} to the specificities of the underlying Galerkin spaces, as long as they are polynomial on each element. For the temporal discretization, we will consider arbitrary order discontinuous Galerkin time-stepping; see Section \ref{sec: A fully discrete scheme}.

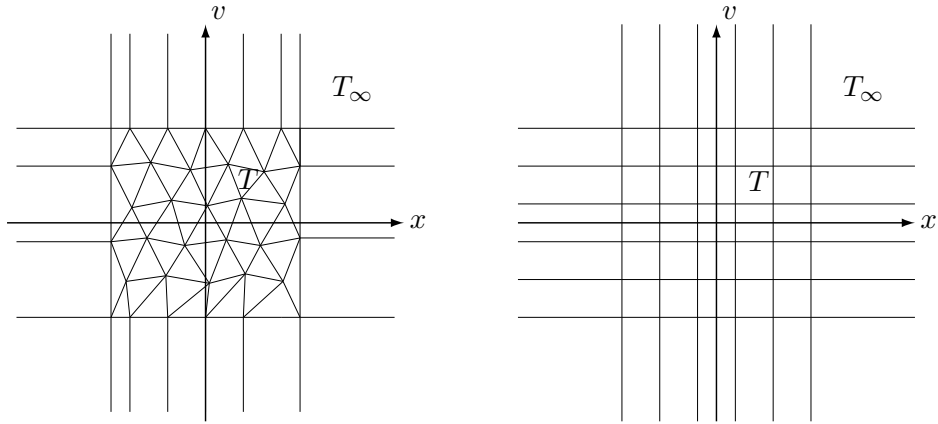
\begin{figure}[h]
	\centering
	\setlength{\unitlength}{0.25mm}
\begin{picture}(220,220)(-110,-110)
	\linethickness{0.35pt}

	\Line(-50,-50)(-40,-50)
	\Line(-40,-50)(-20,-50)
	\Line(-20,-50)(  0,-50)
	\Line(  0,-50)( 20,-50)
	\Line( 20,-50)( 40,-50)
	\Line( 40,-50)( 50,-50)

	\Line(50,-50)(50,-8)
	\Line(50,-8)(50,30)
	\Line(50,30)(50,50)

	\Line(50,50)(40,50)
	\Line(40,50)(20,50)
	\Line(20,50)(0,50)
	\Line(0,50)(-20,50)
	\Line(-20,50)(-40,50)
	\Line(-40,50)(-50,50)

	\Line(-50,50)(-50,30)
	\Line(-50,30)(-50,-10)
	\Line(-50,-10)(-50,-50)

	\Line(-42,-31)(-21,-28)
	\Line(-21,-28)(  2,-32)
	\Line(  2,-32)( 21,-27)
	\Line( 21,-27)( 41,-31)

	\Line(-50,-10)(-31,-8)
	\Line(-31,-8)(-11,-12)
	\Line(-11,-12)(11,-9)
	\Line(11,-9)(29,-12)
	\Line(29,-12)(50,-8)

	\Line(-39,8)(-18,12)
	\Line(-18,12)(1,9)
	\Line(1,9)(19,13)
	\Line(19,13)(42,10)

	\Line(-50,30)(-29,32)
	\Line(-29,32)(-8,28)
	\Line(-8,28)(12,31)
	\Line(12,31)(31,27)
	\Line(31,27)(50,30)

	\Line(-50,-50)(-42,-31)

	\Line(-40,-50)(-42,-31)
	\Line(-40,-50)(-21,-28)

	\Line(-20,-50)(-21,-28)
	\Line(-20,-50)(2,-32)

	\Line(0,-50)(2,-32)
	\Line(0,-50)(21,-27)

	\Line(20,-50)(21,-27)
	\Line(20,-50)(41,-31)

	\Line(50,-50)(41,-31)

	\Line(-42,-31)(-50,-10)
	\Line(-42,-31)(-31,-8)

	\Line(-21,-28)(-31,-8)
	\Line(-21,-28)(-11,-12)

	\Line(2,-32)(-11,-12)
	\Line(2,-32)(11,-9)

	\Line(21,-27)(11,-9)
	\Line(21,-27)(29,-12)

	\Line(41,-31)(29,-12)
	\Line(41,-31)(50,-8)

	\Line(-50,-10)(-39,8)
	\Line(-31,-8)(-39,8)

	\Line(-31,-8)(-18,12)
	\Line(-11,-12)(-18,12)

	\Line(-11,-12)(1,9)
	\Line(11,-9)(1,9)

	\Line(11,-9)(19,13)
	\Line(29,-12)(19,13)

	\Line(29,-12)(42,10)
	\Line(50,-8)(42,10)

	\Line(-39,8)(-50,30)
	\Line(-39,8)(-29,32)

	\Line(-18,12)(-29,32)
	\Line(-18,12)(-8,28)

	\Line(1,9)(-8,28)
	\Line(1,9)(12,31)

	\Line(19,13)(12,31)
	\Line(19,13)(31,27)

	\Line(42,10)(31,27)
	\Line(42,10)(50,30)

	\Line(-50,30)(-40,50)

	\Line(-29,32)(-40,50)
	\Line(-29,32)(-20,50)

	\Line(-8,28)(-20,50)
	\Line(-8,28)(0,50)

	\Line(12,31)(0,50)
	\Line(12,31)(20,50)

	\Line(31,27)(20,50)
	\Line(31,27)(40,50)

	\Line(50,30)(40,50)
	\Line(50,30)(50,50)

	\Line(-50,-50)(-100,-50)
	\Line(-50,-10)(-100,-10)
	\Line(-50, 30)(-100, 30)
	\Line(-50, 50)(-100, 50)

	\Line(50,-50)(100,-50)
	\Line(50,-8 )(100,-8 )
	\Line(50, 30)(100, 30)
	\Line(50, 50)(100, 50)

	\Line(-50,-50)(-50,-100)
	\Line(-40,-50)(-40,-100)
	\Line(-20,-50)(-20,-100)
	\Line(  0,-50)(  0,-100)
	\Line( 20,-50)( 20,-100)
	\Line( 50,-50)( 50,-100)

	\Line(-50,50)(-50,100)
	\Line(-40,50)(-40,100)
	\Line(-20,50)(-20,100)
	\Line(  0,50)(  0,100)
	\Line( 20,50)( 20,100)
	\Line( 40,50)( 40,100)
	\Line( 50,50)( 50,100)

		\linethickness{0.55pt}
	\put(-105,0){\vector(1,0){210}}
	\put(0,-105){\vector(0,1){210}}

	\put(108,-3){$x$}
	\put(3,108){$v$}

	\put(17,18){$T$}
	\put(67,67){$T_\infty$}

\end{picture}
\hspace{1cm}
\begin{picture}(220,220)(-110,-110)
	\linethickness{0.35pt}

	\Line(-50,-50)(-50,50)
	\Line(-30,-50)(-30,50)
	\Line(-10,-50)(-10,50)
	\Line( 10,-50)( 10,50)
	\Line( 30,-50)( 30,50)
	\Line( 50,-50)( 50,50)

	\Line(-50,-50)(50,-50)
	\Line(-50,-30)(50,-30)
	\Line(-50,-10)(50,-10)
	\Line(-50, 10)(50, 10)
	\Line(-50, 30)(50, 30)
	\Line(-50, 50)(50, 50)

	\Line(-50,-50)(-105,-50)\Line(50,-50)(105,-50)
	\Line(-50,-30)(-105,-30)\Line(50,-30)(105,-30)
	\Line(-50,-10)(-105,-10)\Line(50,-10)(105,-10)
	\Line(-50, 10)(-105, 10)\Line(50, 10)(105, 10)
	\Line(-50, 30)(-105, 30)\Line(50, 30)(105, 30)
	\Line(-50, 50)(-105, 50)\Line(50, 50)(105, 50)

	\Line(-50,-50)(-50,-105)\Line(-50,50)(-50,105)
	\Line(-30,-50)(-30,-105)\Line(-30,50)(-30,105)
	\Line(-10,-50)(-10,-105)\Line(-10,50)(-10,105)
	\Line( 10,-50)( 10,-105)\Line( 10,50)( 10,105)
	\Line( 30,-50)( 30,-105)\Line( 30,50)( 30,105)
	\Line( 50,-50)( 50,-105)\Line( 50,50)( 50,105)

	\linethickness{0.55pt}
	\put(-105,0){\vector(1,0){210}}
	\put(0,-105){\vector(0,1){210}}

	\put(108,-3){$x$}
	\put(3,108){$v$}

	\put(17,17){$T$}
	\put(67,67){$T_\infty$}
\end{picture}
\caption{Illustration of meshes comprising of finite and `infinite' elements for $d=1$. Left panel: an `$(x,v)$-unstructured' mesh. Right panel: an `$(x,v)$-tensor-product' mesh.}
\label{fig:meshes}
\end{figure}
We consider \eqref{eq:FP_u_exp} posed over $(0,t_f]\times \mathbb{R}^d\times \mathbb{R}^d$. This requires special treatment. To that end, we subdivide $\mathbb{R}^d\times \mathbb{R}^d$ into a combination of subdomains of finite and of infinite diameter.  An illustration is given in Figure \ref{fig:meshes} for $d=1$, showing a `$(v,x)$-unstructured' scenario, whereby a mesh is constructed directly in $\mathbb{R}^{2d}$, and a `$(v,x)$-tensor-product' option, with the mesh constructed by tensorising elements $T_v\subset \mathbb{R}^d$ and $T_x\subset \mathbb{R}^d$. Note that in the latter case for $d>1$, it is by all means possible to have \emph{unstructured}, e.g., simplicial meshes on each bounded domain in $\mathbb{R}^d$. 
Due to the weight $\ex^{-E}$, integrals of polynomials over `infinite' elements ($T_\infty$ in Figure \ref{fig:meshes}) remain bounded.

\subsection{Spatial discretization}
 We consider a generic finite dimensional continuous approximation space $V_h\equiv V_h^p\subset H^1(\mu)$, consisting of piecewise polynomial functions in $\mathbb{R}^{2d}$ of degree $p$. For all possible finite element spaces discussed below, we assume that the global constant functions are included in $V_h$.

Let $T\in\mathcal{T}$ denote a generic element of the domain subdivision $\mathcal{T}$ and $n_v, n_x\in\mathbb{R}^d$ signifying the components of the unit outward normal vector $\mbf{n}=(n_v^T,n_x^T)^T\in \mathbb{R}^d\times\mathbb{R}^d$ to $\partial T$ for the variables $v$ and $x$, respectively. We denote by $\Gamma$ the skeleton of the subdivision $\mathcal{T}$, viz., $\Gamma:=\cup_{T\in\mathcal{T}}\partial T$. Further, let
$\ud \nu:= \ex^{-E}\ud s$ with $\ud s$ the differential of the Hausdorff measure on $\Gamma$.
Also, let $\nx^h$, $\nv^h$, $\nvx^h$ denote the broken (i.e., element-wise) versions of gradients with respect to $x,v, (v,x)$, respectively, subordinate to $V_h$.

We will make use of the jump $\jump{\cdot}_{z}$ and average $\av{\cdot}$ notation across element interfaces: for $W$ a scalar or vector function with well-defined traces on $\Gamma$, we set
\[
\jump{W}_{z} |_F=  \big(W^+|_F\big)\otimes \mbf{n}_z^++\big(W^-|_F\big)\otimes \mbf{n}_z^-,
\qquad
\av{W} |_F= \frac{1}{2}\big( W^+|_F+W^-|_F),
\]
for each interface $F=\partial T^+\cap\partial T^-$ of two adjacent elements $T^+,T^-\in\mathcal{T}$, with $W^+$ and $W^-$ denoting the traces taken from within $T^+$ and $T^-$, and $\mbf{n}_z^\pm:=\mbf{n}_z|_{T^\pm}$ for $z\in\{v_1,\dots, v_d, x_1,\dots,x_d\}$ is the (scalar) component of the normal vector along the $z$-th partial derivative direction. For example, $\jump{W}_{v_j}=W^+|_F \mbf{n}_{v_j}^++W^-|_F\mbf{n}_{v_j}^-$, has the \emph{same} dimension as $W$ and denotes the jump along the $v_j$ variable direction.

Also, the jump operator is defined with $z\in\{v,x\}$ for $W$ a $d$-dimensional vector-valued function with well-defined traces on $\Gamma$. For example, $\jump{W}_{x} |_F=  \big(W^+|_F\big)\otimes n_x^++\big(W^-|_F\big)\otimes n_x^-$ denotes a $d\times d$ dimensional tensor-valued function. Similarly, we define the jump operator for  $2d$-dimensional vector-valued $W$ by removing $z$, i.e, $\jump{W} |_F=  \big(W^+|_F\big)\otimes \mbf{n}^++\big(W^-|_F\big)\otimes \mbf{n}^-$ is a $2d\times 2d$ dimensional tensor-valued function.

Recalling the notation $\mathcal{V}=(- \nabla V^T , v^T)^T$, we first consider the spatially discrete problem as means to present the spatial discretization of model \eqref{eq:hypo_weak_form}. In particular, we seek $U\in V_h$, such that, for every $t\in (0,t_f]$, we have
\begin{equation}\label{eq:method_semi_discrete}
\dip{U_t}{W}+a_h(U,W) +s_h(U,W)= 0,
\end{equation}
for all $W\in V_h$, where
$
\dip{U_t}{W}:=(U_t, W)_{L_2(\mu)}+(\nvx  U_t, \mathcal{A}\nvx  W)_{L_2(\mu)}$,
\begin{equation}\label{def: discrete bilinear form}
\begin{split}
a_h(U,W):=&\  ( (\mathcal{V}\cdot\nvx) U,W)_{L_2(\mu)}+((\mathcal{V}\cdot\nvx^h )\nvx  U, \mathcal{A}\nvx  W)_{L_2(\mu)}
\\
&	+
(\mathcal{B}\nvx  U,\nvx  W)_{L_2(\mu)}
+\sum_{j=1}^d(\nvx^h   U_{v_j} , \mathcal{A}\nvx^h  W_{v_j})_{L_2(\mu)},
\end{split}
\end{equation}
and
$s_h:V_h\times V_h\to \mathbb{R}$ the stabilisation bilinear form:
\begin{equation}\label{def: discrete stabilsationform}
	\begin{aligned}
		s_h(U,W):=&\  -	\int_{\Gamma}\av{\nvx  W}^T\mathcal{A}\jump{\nvx  U}\mathcal{V}\ud \nu
		 -\tau\int_{\Gamma} \av{\nv W}^T\jump{\nv U}_x \nabla V\ud \nu \\
		&-\sum_{j=1}^d\int_{\Gamma}\Big({ \av{\nvx U_{v_j}}^T\mathcal{A}\jump{\nvx  W}_{v_j}}+{ \av{\nvx W_{v_j}}^T\mathcal{A}\jump{\nvx  U}_{v_j}}\\
	 &	\qquad\qquad\quad
		-\sigma\jump{\nvx U}_{v_j}^T{ \mathcal{C}}\jump{\nvx  W}_{v_j}\Big)\ud \nu,
	\end{aligned}
\end{equation}
with $\sigma:\Gamma\to\mathbb{R}$ a non-negative function, henceforth termed as \emph{interior penalty parameter}, $\tau\ge 0$ constant whose precise value will be determined in the proof of (hypo)coercivity below, and the auxiliary matrix
$
\mathcal{C}=
\big({\rm diag}(\alpha I,\gamma I)\big)^{-1}\mathcal{A}^2$,
with ${\rm diag}(A,B)={\small \Big(
	\begin{array}{cc}
		A & 0 \\
		0 & B
	\end{array}
	\Big)}$, for two square matrices $A,B$. The regularity and growth assumptions on the potential $V$ will be stated explicitly below, and are such that all the terms in \eqref{eq:method_semi_discrete} are well-defined. 

 The choice of $s_h$ is crucial in preserving a `discrete' version of the hypocoercivity property of Lemma \ref{lem:hypoco}. In particular, the first term of $s_h$ ensures the coercivity of the third order derivative terms (i.e., the second term on the right-hand side of $a_h$,) upon introduction of the broken gradient.   The terms ${ \av{\nvx U_{v_j}}^T\mathcal{A}\jump{\nvx  W}_{v_j}}$ ensure the consistency of the discretization of the fourth order derivative terms, while ${ \av{\nvx W_{v_j}}^T\mathcal{A}\jump{\nvx  U}_{v_j}}$ are added to symmetrise the discretization of the (self-adjoint) fourth order terms. The terms involving the penalties are also added to ensure the coercivity of the last two mentioned groups of terms.  

 \begin{remark}The second term on the right-hand side of $s_h$ is active only in the case of `$(x,v)$-unstructured' meshes.
 	In the, physically reasonable and common, case of `$(x,v)$-tensor-product' meshes, i.e., meshes whose elements $T\in\mathcal{T}$ defined as tensor products of $d$-dimensional simplices or boxes, viz., $T=T_v\times T_x$, with $T_v,T_x\subset\mathbb{R}^d$,  we have $\jump{\nv U}_x=\mbf{0}$, resulting to the second term in $s_h$ to vanish. This term arises upon application of Lemma \ref{lem:villani_control_discrete} below, which is a discretized version of {\cite[Lemma A.24]{villani}}, given also in Lemma  \ref{lem:villani_control} above.
 \end{remark}

Note that $s_h(u,W)=0$ for $u\in H^3(\mu)$ and $W\in V_h$, thus, preserving the consistency with respect to the exact problem.

\begin{remark}[mass conservation of semidiscrete scheme]\label{rem:discrete_cons}
The discrete version of \eqref{eq:fp_cons_mass} holds. Indeed, since $V_h$ contains constants, we set $W=1$ into \eqref{eq:method_semi_discrete}, to deduce
	\begin{equation}\label{eq:mass_cons}
		\int U\ud \mu= 	\int u_0\ud \mu=\bar{u},\quad t\in[0,t_f].
	\end{equation}
\end{remark}

\section{Inverse estimates with exponential weights}\label{sec: Inverse estimates with exponential weights}


The proof of (hypo)coerci\-vity for the interior penalty terms in $s_h$ will require the estimation of the third and fourth terms on the right-hand side of \eqref{def: discrete stabilsationform}. To that end, we require, so-called, trace inverse inequalities of the form $\|W\|_{L_2(\nu,F)}\le C(T,p)\|W\|_{L_2(T,\mu)}$ in exponentially weighted norms, with $F$ face of an element $T$, holding for all polynomials $W$ up to degree $p$, for some constant $C(T,p)$ depending on the shape and size of $T$ and $p$. In the present context of weighted norms, we are required to prove new trace inverse inequalities involving exponential weights. These may be of independent interest.

We first prove an one-dimensional estimate for Gaussian-like weights that will be used for the derivation of the multivariate versions.
\begin{lemma}\label{them:weighted_general_1d_full}
	For any $s\in \mathcal{P}_p([0,1])$, $p\in\mathbb{N}$, $A\ge0$, and $B\in \mathbb{R}$, we have the bound
	\begin{equation}\label{basic_unit_interval}
		s^2(0)
		\le  (2\lfloor (A+B_+)/4\rfloor +p+3)^2\int_0^1s^2(x) 	\ex^{-(Ax^2+Bx)/2}\ud x,
	\end{equation}
	with $\lfloor\cdot\rfloor$ the integer part and $(\cdot)_+$ the positive part of a real number, respectively.
\end{lemma}
\begin{proof}

	Let $g:\mathbb{R}\to\mathbb{R}$ with
	$
	g(x) =\tilde{g}^m(x)\ex^{(Ax^2+Bx)/4}$, whereby $\tilde{g}(x):=1-(Ax^2+Bx)/(4m)$,
	for $m\in \mathbb{N}$, to be determined below; cf., also \cite{erdelyi89} for a use of a similar function used in the proof of weighted Markov-type inequalities.
	We compute
	\[
		g'(x)
		= -(2Ax+B)(Ax^2+Bx)\tilde{g}^{m-1}(x)\ex^{(Ax^2+Bx)/4}/(16m)
.
\]
	First, we consider the case $B\ge 0$. Then, $(2Ax+B)(Ax^2+Bx)\ge 0$ for $x\in[0,1]$. If, additionally,  $m:= \lfloor (A+B)/4\rfloor+1$, we have $0\le \tilde{g}(x)$ and, thus,
	$g$ is decreasing. Hence, $g(x)\le g(0)$, for $x\in [0,1]$,
	$
	\ex^{-(Ax^2+Bx)/2}\ge \tilde{g}^{2m}(x)
	$.
	Now, let $s\in \mathcal{P}_p([0,1])$.  Then, a standard one-dimensional trace inverse inequality of the form $|w(a)|^2 \le (p+1)^2/(b-a)\norm{w}{L_2((a,b))}^2$, for $w\in\mathcal{P}_p((a,b))$, see, e.g., \cite{warburton2003constants}, yields
	\[
	\begin{aligned}
		s^2(0)=\big(s(0) \tilde{g}^m(0)\big)^2\le&\  (2\lfloor (A+B)/4\rfloor +p+3)^2\int_0^1s^2(x) \tilde{g}^{2m}(x)\ud x\\
		\le&\  (2\lfloor (A+B)/4\rfloor +p+3)^2\int_0^1s^2(x) 	\ex^{-(Ax^2+Bx)/2}\ud x.
	\end{aligned}
	\]
For the case $B<0$, we note that $\ex^{-Bx}\ge1$ for $x\in[0,1]$. Now, starting from \eqref{basic_unit_interval} for $B=0$, we have $ s^2(0)
\le  (2\lfloor A/4\rfloor +p+3)^2\int_0^1s^2(x) 	\ex^{-Ax^2/2}\ud x $ and, since $	\ex^{-Ax^2/2}\le 	\ex^{-(Ax^2+Bx)/2}$, the result follows.
\end{proof}

We now use the above to show the corresponding trace inverse inequality for the Gaussian weight on an arbitrary line segment $[\bx_0,\bx_1]$, with endpoints $\bx_0,\bx_1\in\mathbb{R}^d$. Let  $\mbf{r}(t):=\bx_0+t \mbf{b}$,   $t\in [0,1]$, with $\mbf{b} :=\bx_1-\bx_0$ parametrisation of $[\bx_0,\bx_1]$. Then, $|\mbf{r}(t)|^2=|\bx_0|^2+t^2 |\mbf{b}|^2+2t\bx_0\cdot \mbf{b}$. Also, $|\mbf{r}'(t)|=|\mbf{b}|$. Let also $v\in\mathcal{P}_p(K)$, for an open set $K\subset\mathbb{R}^d$, such that $[\bx_0,\bx_1]\subset K$.  Set $s(t):=v(\mbf{r}(t))$, $t\in [0,1]$. Then, $s\in \mathcal{P}_p([0,1])$. Defining
$
C(\mbf{b},\bx_0,p):= (2\lfloor (|\mbf{b}|^2+(2\bx_0\cdot \mbf{b})_+)/4\rfloor +p+3)^2$,
we employ \eqref{basic_unit_interval} with $A=|\mbf{b}|^2$ and $B=2\bx_0\cdot \mbf{b}$,  to get
\begin{equation}
	\label{magic_one}
	\begin{aligned}
		v^2(\bx_0)\ex^{-|\bx_0|^2/2}
		\le&\   C(\mbf{b},\bx_0,p)\int_0^1s^2(t) 	\ex^{-\big( |\bx_0|^2+t^2|\mbf{b}|^2+2t\bx_0\cdot \mbf{b}\big)/2}\ud t\\
		=&\   C(\mbf{b},\bx_0,p) |\mbf{b}|^{-1} \int_0^1v^2(\mbf{r}(t))	\ex^{- |\mbf{r}(t)|^2/2}|\mbf{r}'(t)|\ud t\\
		=&\   C(\mbf{b},\bx_0,p) |\mbf{b}|^{-1} \int_{[\bx_0,\bx_1]}v^2(\bx)	\ex^{-|\bx|^2/2}\ud\bx.
	\end{aligned}
\end{equation}

From this we can prove a general estimate on simplicial and/or box-type elements $T\in\mathcal{T}$ via tensor-product arguments.

\begin{lemma}\label{lem:inv_est_Gaussian}
	Let $v\in\mathcal{P}_p(T)$, for a simplicial or box-type element $T\in\mathcal{T}$.  Set $\delta_T:=\max\{|\bx|:\bx\in \bar{T}\}$. Then, for each face $F\subset \partial T$, we have the inequality
	\begin{equation}\label{inv_est_Gaussian}
	\int_{F}v^2(\bs)\ex^{-|\bs|^2/2}\ud \bs \le
	c_{\rm inv}(\delta_T) \frac{p^2}{h_T} \int_{T}  v^2(\bx)	\ex^{-|\bx|^2/2}\ud\bx,
	\end{equation}
	 with
	$
	c_{\rm inv}(\delta_T):=M_T^Fh_T \big( h_T^2+2\delta_T h_T+2p+6\big)^2/(4\rho_T^{}p^2\min_\theta\sin \theta),
	$  with $\theta$ denoting the angles of $T$, $\rho_T^{}$ the radius of the largest inscribed ball of $T$, and $M_T^F=2$ for $d=2$, while $M_T^F=7$ for $d=3$.
\end{lemma}
\begin{proof}
	If $T$ is simplicial, each of its faces $F\subset \partial T$ is subdivided into a union of non-overlapping subfaces $F_i$, $i=1,\dots, M_T^F$, $M_T^F\in\mathbb{N}$,  as follows. For $d=2$, we take $M_T^F=2$, and we have two subfaces $F_1$ and $F_2$ of $F$ having equal length separated by the midpoint of $F$; we refer to Figure \ref{fig:oblique} for an illustration, as well as for the respective construction for $d=3$, having $M_T^F=7$.

	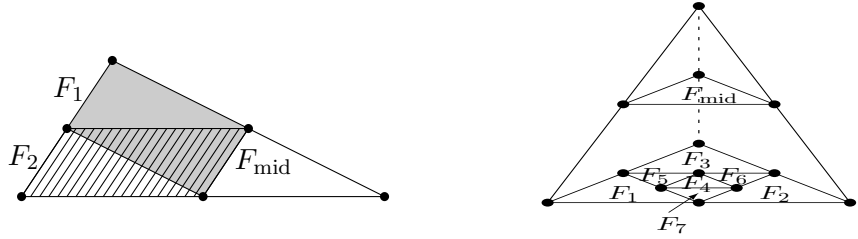
\begin{figure}[h]
		\centering
		\setlength{\unitlength}{0.6mm}
		\begin{picture}(100,50)

		\color[gray]{0.8}
			\polygon*(30,40)(60,25)(50,10)(20,25)

			\color{black}

			\Line(10,10)(90,10)
			\Line(10,10)(30,40)
			\Line(30,40)(90,10)

			\Line(20,25)(60,25)
			\Line(60,25)(50,10)
			\Line(50,10)(20,25)

			\put(10,10){\circle*{2}}
			\put(90,10){\circle*{2}}
			\put(30,40){\circle*{2}}

			\put(20,25){\circle*{2}}
			\put(60,25){\circle*{2}}
			\put(50,10){\circle*{2}}

			\put(7,17){$F_2$}

			\put(17,32){$F_1$}

			\multiput(10,10)(2,0){21}{\Line(0,0)(10,15)}

			\put(57,16){$F_{\rm mid}$}
		\end{picture}\hspace{0.5cm}
	\scalebox{1}[0.65]{
		\begin{picture}(80,70)
		\setlength{\unitlength}{0.8mm}

				\put(10,10){\circle*{2}}
				\put(60,10){\circle*{2}}
				\put(35,25){\circle*{2}}
				\put(35,60){\circle*{2}}

				\Line(10,10)(60,10)
				\Line(60,10)(35,25)
				\Line(35,25)(10,10)

				\Line(35,60)(10,10)
				\Line(35,60)(60,10)

				\dashline{1}(35,60)(35,25)

			\put(35,10){\circle*{2}}
			\put(47.5,17.5){\circle*{2}}
			\put(22.5,17.5){\circle*{2}}

			\Line(35,10)(47.5,17.5)
			\Line(47.5,17.5)(22.5,17.5)
			\Line(22.5,17.5)(35,10)

		\put(22.5,35){\circle*{2}}
		\put(47.5,35){\circle*{2}}
		\put(35,42.5){\circle*{2}}

		\Line(22.5,35)(47.5,35)
		\Line(47.5,35)(35,42.5)
		\Line(35,42.5)(22.5,35)

		\put(41.25,13.75){\circle*{2}}
		\put(35,17.5){\circle*{2}}
		\put(28.75,13.75){\circle*{2}}

		\Line(41.25,13.75)(35,17.5)
		\Line(35,17.5)(28.75,13.75)
		\Line(28.75,13.75)(41.25,13.75)

		\put(20,11){$F_1$}
		\put(45,11){$F_2$}
		\put(32,19.5){$F_3$}
		\put(32,13.5){$\tiny F_4$}
		\put(25,15.5){$\tiny F_5$}
		\put(38,15.5){$\tiny F_6$}
		\put(28,3){$\tiny F_7$}
		\put(30,7.5){\vector(1,1){5}}

		\put(32,36){$F_{\rm mid}$}
		\end{picture}
	}

		\caption{Left panel: subdivision of face  $F$ into $F_1$ and $F_2$, that are bases of the oblique prisms (parallelograms) $\underline{T}_1$ (shaded) and $\underline{T}_2$ (hatched).  Right panel: subdivision of face  $F$ (base of tetrahedron) into $F_i$, $i=1,\dots,7$, that are bases of respective oblique prisms $\underline{T}_i\subset T$. (The respective oblique prisms are omitted.)}
		\label{fig:oblique}
	\end{figure}

Thus, each $F_i$ is one of the two bases of the oblique prism in $d$ dimensions:
\[
\underline{T}_i:=\{\bx=\bx_{F_i}+t\mbf{b}_{F_i}: \bx_{F_i}\in F_i, \mbf{b}_{F_i}\in\mathbb{R}^d, t\in(0,1]\} \subset T,
\]
with the second base given by a subset of the $(d-1)$-dimensional simplex parallel to the original base and located at half the total height of the simplex, denoted by $F_{\rm mid}$ in Figure~\ref{fig:oblique}. 
If $T$ is a box-type element, i.e., a linearly mapped hypercube,  then it is of the form $\underline{T}_i$, so the proof below still holds.

Now consider the shear transformation $\mathcal{R}_{\underline{T}_i}:\hat{\underline{T}}_i\to \underline{T}_i$, with $\hat{\underline{T}}_i$ denoting a right-angle prism with the same axis $\mbf{b}_{F_i}$ and having one identical face parallel to $\mbf{b}_{F_i}$ with $\underline{T}_i$ with axis $\mbf{b}_{F_i}$. In other words, $\underline{T}_i$ is constructed by a shear transformation  of the right-angle prism $\hat{\underline{T}}_i$. Thus, $\mathcal{R}_{\underline{T}_i}$ can be represented as the action of a matrix $R_{\underline{T}_i}\in \mathbb{R}^{d\times d}$ with $|\det (R_{\underline{T}_i})|=1$. Note also that $ R_{\underline{T}_i}\mbf{b}_{F_i}=\mbf{b}_{F_i}$.

Let also $\hat{F}_i$ denote the (pulled-back) base of $\hat{\underline{T}}_i$ so that $\mathcal{R}_{\underline{T}_i}\hat{F}_i=F_i$. Hence, for each $\hat{\bs}\in \hat{F}_i$,  we have $R_{\underline{T}_i}\hat{\bs}=\bs\in F_i$. Then, \eqref{magic_one}  implies
\[
\begin{aligned}
	v^2(R_{\underline{T}_i}\hat{\bs})\ex^{-|R_{\underline{T}_i}\hat{\bs}|^2/2}=	&\ v^2(\bs)\ex^{-|\bs|^2/2}\le
	   C(\mbf{b}_{F_i},\bs,p) |\mbf{b}_{F_i}|^{-1} \int_{[\bs,\bs+\mbf{b}_{F_i}]} v^2(\by)	\ex^{-|\by|^2/2}\ud\by\\
	   \le&\
	   C(\mbf{b}_{F_i},\bs,p) |\mbf{b}_{F_i}|^{-1} \int_{[\hat{\bs},\hat{\bs}+\mbf{b}_{F_i}]} v^2(R_{\underline{T}_i}\hat{\by})	\ex^{-|R_{\underline{T}_i}\hat{\by}|^2/2}\ud \hat{\by},
\end{aligned}
\]
since the shear transformation does not alter length in the direction of $\mbf{b}_{F_i}$. Integrating over $\hat{F}_i$, and taking the maximum on the constants over $\bs\in F_i$ results to
\[
	\int_{\hat{F}_i}	v^2(R_{\underline{T}_i}\hat{\bs})\ex^{-|R_{\underline{T}_i}\hat{\bs}|^2/2}\ud \hat{\bs}
	\le
	 \max_{\bs\in F_i}  C(\mbf{b}_{F_i},\bs,p) |\mbf{b}_{F_i}|^{-1} \int_{\hat{\underline{T}}_i}  v^2(R_{\underline{T}_i}\hat{\bx})	\ex^{-|R_{\underline{T}_i}\hat{\bx}|^2/2}\ud\hat{\bx}
\]
Changing variables back gives
\begin{equation}\label{basic_inverse_mapped}
\begin{aligned}
	\frac{|\hat{F}_i|}{|F_i|}\int_{F_i}	v^2(\bs)\ex^{-|\bs|^2/2}\ud \bs
	\le
	&\ \max_{\bs\in F_i}  C(\mbf{b}_{F_i},\bs,p) |\mbf{b}_{F_i}|^{-1} \int_{\underline{T}_i}  v^2(\bx)	\ex^{-|\bx|^2/2}\ud\bx,
\end{aligned}
\end{equation}
since $|\det (R_{\underline{T}_i})|=1$, with $|\hat{F}_i|,|F_i|$ denoting the respective $(d-1)$-dimensional volumes. Since $\hat{F}_i$ and $F_i$ are related by a shear transformation, their $(d-1)$-dimensional volumes will be related by the sine of the smallest angle between them. Thus, on the element $T$ we have
\begin{equation}\label{shape_reg2}
	|\hat{F}_i|/|F_i|\ge \min_\theta\sin \theta,
\end{equation}
for all $i=1,\dots,M_T^F$ with $\theta$ an angle of $T$. Also,  that the length of each prism axis vector $\mbf{b}_{F_i}$ is, of course, bounded by $h_T={\rm diam}(T)$. Moreover, since $F_i$ for all element faces $F\subset \partial T$, we have
\begin{equation}\label{shape_reg}
	\rho_T^{}\le |\mbf{b}_{F_i}|\le \frac{h_T}{2},
\end{equation}
 with $\rho_T^{}>0$ denoting the inscribed radius of $T$.
Then, \eqref{shape_reg2}  and \eqref{shape_reg} imply
\[
\begin{aligned}
\frac{	C(\mbf{b}_{F_i},\bs,p)  |\mbf{b}_{F_i}|^{-1}}{\sin \theta}
	\le&\  
	 \frac{c_{\rm inv}(\delta_T)}{M_T^F}\frac{p^2}{h_T}.
\end{aligned}
\]
Summing up for $i=1,\dots,M_T^F$, we arrive at the desired bound.
\end{proof}
\begin{remark}
	Note that for $h_T\sim\delta_T^{-1}$, $c_{\rm inv}(\delta_T)\le C$ for a constant $C$ depending only on the geometry of $T$.
\end{remark}

\begin{lemma}\label{inf_strip_Gaussian}
For `infinite' elements with base $F$, defined by
\begin{equation}\label{inf_prism}
T=\{\bx=\bx_{F}+t\mbf{b}_{F}: \bx_{F}\in F,\ \mbf{b}_{F}\in\mathbb{R}^d, t\in(0,\infty)\},
\end{equation}
we have the trace inverse inequality
	\begin{equation}\label{inv_est_Gaussian_inf}
	\int_{F}v^2(\bs)\ex^{-|\bs|^2/2}\ud \bs \le
	\tilde{c}_{\rm inv}(\delta_T) p^{\frac{3}{2}} \int_{T}  v^2(\bx)	\ex^{-|\bx|^2/2}\ud\bx,
\end{equation}
with $
\tilde{c}_{\rm inv}(\delta_T):=(3p+2\delta_T\sqrt{p}+6)^2/(4p^2\min_\theta\sin\theta)$, with $\theta$ the angles of $T$ as \eqref{shape_reg2}.
\end{lemma}
\begin{proof}
The proof follows by application of  the inverse estimate \eqref{basic_inverse_mapped} from the proof Lemma \ref{lem:inv_est_Gaussian} to a bounded subdomain of the infinite element of axis length $\sqrt{p}$ containing the same boundary face as the infinite element.

	First we discuss the case of an infinite `strip' element with bounded base $F$.
	Consider the prism $\tilde{T}:=\{\bx=\bx_{F}+t\mbf{b}_{F}: \bx_{F}\in F, \mbf{b}_{F}\in\mathbb{R}^d, t\in(0,1]\}$, for which we trivially have $\tilde{T}\subset T$. Select $\mbf{b}_F$ such that $|\mbf{b}_F|^2=p$. Then, assuming that $|F|<\infty$, we employ, as above,  a shear transformation $\mathcal{R}_{\tilde{T}}:\hat{T}\to\tilde{T}$ with $\hat{T}$ denoting the corresponding right-angle prism. Let also $\hat{F}$ denote the base of $\hat{T}$, so that $\mathcal{R}_{\tilde{T}}\hat{F}=F$. Then, \eqref{basic_inverse_mapped} and \eqref{shape_reg2} imply
	\begin{equation}\label{magic_inf}
\sin\theta 	\int_F	v^2(\bs)\ex^{-|\bs|^2/2} \ud \bs \le
	\max_{\bs\in F}C(\mbf{b}_{F},\bs,p) |\mbf{b}_{F}|^{-1} \int_{\tilde{T}}v^2(\bx)	\ex^{-|\bx|^2/2}\ud\bx,
	\end{equation}
	with $\theta$ the angle of the `shear movement'.
	The domain of integral on the right-hand side of the last estimate can be majorised  to $T$, while the choice of $|\mbf{b}_F|=\sqrt{p}$ gives
	$
	C(\mbf{b}_{F},\bx_{F},p) (|\mbf{b}_{F}|\sin \theta)^{-1}\le \tilde{c}_{\rm inv}(\delta_T) p^{\frac{3}{2}}.
	$

	The estimate \eqref{magic_inf} also holds for elements with infinite base $F$ having $|F|=+\infty$ also, e.g., when $T=\mathbb{R}^d_+$ and $F=\{0\}\times\mathbb{R}^{d-1}_+$; see $T_\infty$ in Figure \ref{fig:meshes} for an illustration. In this case, we consider `strip' subdomains $T_m\subset T$, with bounded bases $F_m\subset F$ and $T_m\to T$, with $|F_m|<\infty$, $m\in\mathbb{N}$, so that $F_m\to F$ and $T_m\to T$, as $m\to \infty$. Then, we can apply \eqref{magic_inf}, on each strip $T_m$ with base $F_m$, noting that $|\hat{F}_m|/|F_m|\ge \min_\theta\sin\theta$, uniformly as $m\to\infty$ and, thus, \eqref{magic_inf} holds, giving the same bound \eqref{inv_est_Gaussian_inf}.
\end{proof}

\begin{remark} The last result holds also for `infinite wedge' elements of any angle, such as $T_\infty$ in Figure \ref{fig:meshes}.  Also, note the improved dependence on the polynomial degree $p$ for \eqref{inv_est_Gaussian_inf} compared to \eqref{inv_est_Gaussian}. Furthermore, the result holds for $d\ge 4$, for appropriate $M_T^F\in\mathbb{N}$.
\end{remark}

Next, we also prove inverse inequalities for powers of Gaussian-like weights.

\begin{lemma}\label{them:weighted_general_1d_powers}
	For any $s\in \mathcal{P}_p([0,1])$, $p\in\mathbb{N}$, $k\in\mathbb{N}$,  for $A,B,C\ge0$, we have the bound
	\begin{equation}\label{basic_unit_interval_powers}
		s^2(0)\ex^{-C^k/(2k)} \le  (2k\lfloor (A+B)^k/(4k)\rfloor +p+3)^2\int_0^1s^2(x) 	\ex^{-(Ax^2+Bx+C)^k/(2k)}\ud x.
	\end{equation}
\end{lemma}
\begin{proof}

	Redefining $g:\mathbb{R}\to\mathbb{R}$ with
	$
	g(x) =\tilde{g}^m(x)\ex^{(Ax^2+Bx+C)^k/(4k)}$,
	 whereby
$
	\tilde{g}(x):=1-(Ax^{2}+Bx +C)^k/(4mk)$,
	for $m\in \mathbb{N}$, to be determined below.
	We compute
	\[
	g'(x)
	= -(2Ax+B)(Ax^2+Bx+C)^{2k-1}\tilde{g}^{m-1}(x)\ex^{(Ax^2+Bx+C)^k/(4k)}/(16km)
	.
	\]
	Since $A,B,C\ge 0$, setting  $m:= \lfloor (A+B)^k/{4k}\rfloor+1$, we have $0\le \tilde{g}(x)$ for all $x\in[0,1]$ and, thus,
	$g$ is decreasing. Hence, $g(x)\le g(0)$, for $x\in [0,1]$ and, so,
	$\ex^{-(Ax^2+Bx+C)^k/(2k)}\ge \tilde{g}^{2m}(x)	\ex^{-C^k/(2k)}$
	. Let $s\in \mathcal{P}_p([0,1])$.  A standard one-dimensional trace inverse inequality \cite[Theorem 2]{warburton2003constants} gives
	\[
	\begin{aligned}
		s^2(0)	\ex^{-C^k/(2k)}=&\ s^2(0) \tilde{g}^{2m}(0)	\ex^{-C^k/(2k)}\\
		\le&\  (2k\lfloor (A+B)^k/(4k)\rfloor +p+3)^2\int_0^1s^2(x) \tilde{g}^{2m}(x)	\ex^{-C^k/(2k)}\ud x,
	\end{aligned}
	\]
	and the result follows by combing the last two inequalities.
\end{proof}

Using the last bound to prove the respective inverse inequality for an arbitrary line segment $[\bx_0,\bx_1]$, with endpoints $\bx_0,\bx_1\in\mathbb{R}^d$, again set $\mbf{r}(t)=\bx_0+t \mbf{b}$,   $t\in [0,1]$, with $\mbf{b} :=\bx_1-\bx_0$. As before, let $v\in\mathcal{P}_p(K)$, for an open set $K\subset\mathbb{R}^d$, such that $[\bx_0,\bx_1]\subset K$ and $s(t)=v(\mbf{r}(t))$, $t\in [0,1]$. Defining
\[
C(k,\mbf{b},\bx_0,p):= (2k\lfloor (|\mbf{b}|^2+2|\bx_0\cdot \mbf{b}|+|\bx_0|^2)^k/(4k)\rfloor +p+3)^2,
\]
we employ \eqref{basic_unit_interval_powers} with $A=|\mbf{b}|^2$, $B=2|\bx_0\cdot \mbf{b}|$, and $C=|\bx_0|^2$,  to get
\[
	\begin{aligned}
		v^2(\bx_0)\ex^{-|\bx_0|^{2k}/(2k)}
		=&\ s^2(0)\ex^{-|\bx_0|^{2k}/(2k)}\\
		\le&\   C(k,\mbf{b},\bx_0,p)\int_0^1s^2(t) \ex^{-\big(t^2 |\mbf{b}|^2+2t|\bx_0\cdot \mbf{b}|+|\bx_0|^2\big)^k/(2k)}\ud t\\
		\le&\   C(k,\mbf{b},\bx_0,p)\int_0^1v^2(\mbf{r}(t)) 	\ex^{- |\mbf{r}(t)|^{2k}/(2k)}\ud t\\
		=&\   C(k,\mbf{b},\bx_0,p) |\mbf{b}|^{-1} \int_{[\bx_0,\bx_1]}v^2(\bx)	\ex^{-|\bx|^{2k}/(2k)}\ud\bx.
	\end{aligned}
\]
For the general inverse estimate for elements $T\in\mathcal{T}$, we work as before.

\begin{lemma}\label{lem:inv_est_powers}
	Let $v\in\mathcal{P}_p(T)$, for a simplicial or box-type element $T\in\mathcal{T}$.  Set $\delta_T:=\max\{|\bx|:\bx\in \bar{T}\}$. Then, for each face $F$ of $T$, the following estimate
	\begin{equation}\label{inv_est_powers}
		\int_{F}v^2(\bs)\ex^{-|\bs|^{2k}/(2k)}\ud \bs \le
		c_{\rm inv}(k,\delta_T) \frac{p^2}{h_T} \int_{T}  v^2(\bx)	\ex^{-|\bx|^{2k}/(2k)}\ud\bx,
	\end{equation}
	holds, with
	$
	c_{\rm inv}(k,\delta_T):=M_T^Fh_T \big( (h_T+\delta_T)^{2k}+2p+6\big)^2/(4\rho_T^{}p^2\min_\theta\sin \theta),
	$ and $\theta$ denoting the angles of $T$ and $\rho_T^{}$ the radius of the largest inscribed ball of $T$, and $M_T^F$ as above.
\end{lemma}

\begin{proof}
Working as in the proof of Lemma \ref{lem:inv_est_Gaussian}, we arrive at
	\[
	\sin\theta\!	\int_{F_i}\!v^2(\bs)\ex^{-|\bs|^{2k}/(2k)}\ud \bs \le
		\! \max_{i=1,\dots, M_T^F}C(k,\mbf{b}_{F_i},\bx_{F_i},p)  |\mbf{b}_{F_i}|^{-1}\!\int_{\underline{T}_i}\!  v^2(\bx)	\ex^{-|\bx|^{2k}/(2k)}\ud\bx.
	\]
	Then, \eqref{shape_reg} implies
	\[
	\begin{aligned}
		\frac{C(k,\mbf{b}_{F_i},\bx_{F_i},p)  |\mbf{b}_{F_i}|^{-1}}{\sin \theta}
		\le&\  \frac{1}{\sin \theta\rho_T}\Big( \frac{ (h_T+\delta_T)^{2k}}{2} +p+3\Big)^2\le  \frac{c_{\rm inv}(k,\delta_T)}{M_T^F}\frac{p^2}{h_T}.
	\end{aligned}
	\]
	Summing up for $i=1,\dots,M_T^F$, we arrive at the desired bound.
\end{proof}

\begin{lemma}
	For `infinite' elements with base $F$ of the form \eqref{inf_prism},
we have
	\begin{equation}\label{inv_est_Gaussian_inf2}
		\int_{F}v^2(\bs)\ex^{-|\bs|^{2k}/(2k)}\ud \bs \le
		\tilde{c}_{\rm inv}(k,\delta_T) p^{2-1/(2k)} \int_{T}  v^2(\bx)	\ex^{-|\bx|^{2k}/(2k)}\ud\bx,
	\end{equation}
	with  $
	\tilde{c}_{\rm inv}(k,\delta_T):=\big(  (p^{1/(2k)}+\delta_T)^{2k} +2p+6\big)^2/(4p^2\min_\theta\sin\theta)$.
\end{lemma}
\begin{proof}
	Working completely analogously to the proof of \eqref{magic_inf},
	majorising the integration domains on right-hand side to $T$, and observing that the choice of $|\mbf{b}_F|=p^{1/(2k)}$ gives
	$
(\sin\theta)^{-1}	C(k,\mbf{b}_{F},\bx_{F},p) |\mbf{b}_{F}|^{-1}\le
	\tilde{c}_{\rm inv}(k,\delta_T) p^{2-1/(2k)}$,
	the result follows.
\end{proof}

\begin{remark}\label{rem:sharper}
	We observe a favourable dependence on $\delta_T$ in $c_{\rm inv}(\delta_T)$ of Lemma \ref{lem:inv_est_Gaussian} against $c_{\rm inv}(k,\delta_T)$ of Lemma \ref{lem:inv_est_powers} for $k=1$.
This is due to the small, yet distinct, difference between the inequality we deduce upon setting $k=1$ on \eqref{basic_unit_interval_powers} compared to \eqref{basic_unit_interval}: \eqref{basic_unit_interval} holds for $B<0$ also. The same observation applies to the constants $\tilde{c}_{\rm inv}(\delta_T)$ and $\tilde{c}_{\rm inv}(k,\delta_T)$ of the respective estimates for the `infinite' elements.
\end{remark}

We now discuss trace inverse inequalities with weight $\ex^{-E}$, for $E$ as in \eqref{exponent}.
%
To that end, we consider the family of general potentials satisfying  \eqref{eq:Vgrowth}.
We begin by observing that when $V(x)=c+|x|^{2k}/(2k)$ for $k\in\mathbb{N}$ and any $c\in\mathbb{R}$, then  \eqref{eq:Vgrowth} holds with $C_0=2k-1$. At the same time, $V(x)=\ex^{|x|^2}$ does not satisfy \eqref{eq:Vgrowth}. Roughly speaking,  \eqref{eq:Vgrowth} is satisfied by at least superlinear and at most exponentially growing $V(x)$ as $|x|\to\infty$.

\begin{theorem}\label{lem:gen_in_est_final}
	Set $G_k(x):=|x|^{2k}/(2k)$, $k\in\mathbb{N}$, and assume that $V$ satisfies
	\begin{equation}\label{growth_V_second}
	G_{k}(x)-c_V^{}\le V(x)\le C_V+	G_{k}(x),\qquad x\in\mathbb{R}^d,
	\end{equation} for given constants  $c_V^{},C_V$.  (In other words, $V(x)=G_{k}(x)+\phi(x)$ with $c_V^{}\le \phi(x)\le C_V$.) Assume further that $T=T_v\times T_x$ with $T_v,T_x\subset \mathbb{R}^d$, with $T_v,T_x$ finite or `infinite' elements of the form \eqref{inf_prism}, as above.  Then, for any $W\in\mathcal{P}_p(T)$, and $k\ge2$, we have
	\begin{equation}\label{inv_est_final}
	\int_{F}W^2\ud \nu  \le C_{\rm inv}(T,p,k)\int_{T}W^2\ud \mu,
	\end{equation}
	with
\[
	C_{\rm inv}(T,p,k):=
	\begin{cases}
	\ex^{c_V^{}+C_V}	c_{\rm inv}(k,\delta_T) p^2h_T^{-1}, & \text{for $F= T_v\times F_x$, and } |T_x|<+\infty; \\
	\ex^{c_V^{}+C_V}	\tilde{c}_{\rm inv}(k,\delta_T) p^{2-1/(2k)}, & \text{for $F= T_v\times F_x$, and $T_x$ as \eqref{inf_prism}} ;\\
	 c_{\rm inv}(\delta_T) p^2h_T^{-1}, & \text{for $F= F_v\times T_x$,  and } |T_v|<+\infty; \\
	\tilde{c}_{\rm inv}(\delta_T) p^{\frac{3}{2}}, & \text{for $F= F_v\times T_x$, and $T_v$ as \eqref{inf_prism}}.
	\end{cases}
\]
For $k=1$, $c_{\rm inv}(k,\delta_T)$ and 	$\tilde{c}_{\rm inv}(k,\delta_T)$ can be replaced by the sharper $c_{\rm inv}(\delta_T)$ and $\tilde{c}_{\rm inv}(\delta_T)$, respectively. Moreover, for the case of Gaussian confining potential $V(x)=G_1(x)$, \eqref{inv_est_final} holds also for $(v,x)$-unstructured elements.
\end{theorem}
\begin{proof}
First, let $F= T_v\times F_x\subset \partial T$, with $F_x\subset \partial T_x$ face of $T_x$. Then, we have
\begin{equation}\label{tensor_inv}
\int_{F}W^2\ud \nu =\int_{T_v}\ex^{-|v|^2/2}\int_{F_x}W^2(v,x)\ex^{-V(x)}\ud s_x \ud v,
\end{equation}
with $\ud s_x$ denoting the surface differential of $F_x$. Then, for every $v\in T_v$, we have
\begin{equation}\label{bound_tensor_inv}
\begin{aligned}
\int_{F_x}W^2(v,x)\ex^{-V(x)}\ud s_x\le&\  \ex^{c_V^{}}\int_{F_x}W^2(v,x)\ex^{-G_{k}(x)}\ud s_x\\
\le&\ \ex^{c_V^{}}C_{\rm inv}^x(T,p,k)\int_{T_x}W^2(v,x)\ex^{-G_{k}(x)}\ud x\\
\le&\ \ex^{c_V^{}+C_V}C_{\rm inv}^x(T,p,k)\int_{T_x}W^2(v,x)\ex^{-V(x)}\ud x,
\end{aligned}
\end{equation}
with $C_{\rm inv}^x(T,p,k):=	c_{\rm inv}(k,\delta_T) p^2h_T^{-1}$ if $|T_x|<+\infty$, or, for $T_x$  `infinite', we set $C_{\rm inv}^x(T,p,k):=	\tilde{c}_{\rm inv}(k,\delta_T) p^{2-1/(2k)}$. Using \eqref{bound_tensor_inv}  on  \eqref{tensor_inv} gives the required bound.

Now, let $F= F_v\times T_x\subset \partial T$, with $F_v\subset \partial T_v$ face of $T_v$. Then, we have
\begin{equation}\label{tensor_inv_v}
	\begin{aligned}
	\int_{F}W^2\ud \nu =&\ \int_{T_x}\ex^{-V(x)}\int_{F_v}W^2(v,x)\ex^{-|v|^2/2}\ud s_v \ud x\\
	\le&\  C_{\rm inv}^v(T,p)\int_{T_x}\ex^{-V(x)}\int_{T_v}W^2(v,x)\ex^{-|v|^2/2}\ud v \ud x,
\end{aligned}
\end{equation}
and the result follows with $C_{\rm inv}^v(T,p):=	c_{\rm inv}(\delta_T) p^2h_T^{-1} $ if $|T_x|<+\infty$, or $C_{\rm inv}^v(T,p):=\tilde{c}_{\rm inv}(\delta_T) p^{\frac{3}{2}} $ for $T_v$ `infinite' element of the form \eqref{inf_prism}.
\end{proof}

\begin{remark}
The growth condition \eqref{growth_V_second} for $V$ is related yet different to \eqref{eq:Vgrowth}. If $\phi \in C^2(\mathbb{R}^d)$ and has bounded gradient and Hessian, then  \eqref{growth_V_second}  implies \eqref{eq:Vgrowth}.
\end{remark}

\section{Numerical hypocoercivity} \label{sec: Numerical hypocoercivity}
Equipped with the above technical results, we now prove the (hypo)coercivity of the Galerkin method \eqref{eq:method_semi_discrete}. We start by considering an extension of Lemma \ref{lem:villani_control}, i.e., \cite[Lemma A.24]{villani}, for element-wise $H^1$-functions.
  \begin{lemma}\label{lem:villani_control_discrete}
 	Let $g\in [H^1(\mu,\mathcal{T})]^d$, with $H^1(\mu,\mathcal{T}):=\{v\in L_2(\mu): \!\int_T\!|\nvx^hv|^2\ud \mu_T\!<\!+\infty, T\in \mathcal{T} \}$, $\mu_T$ the restriction of $\mu$ to $T$, and $V$  satisfying \eqref{eq:Vgrowth}. Then,  we have
 	\begin{equation}\label{eq:villani_control_discrete}
 		\norm{\mathcal{H}(V) g}{L_2(\mu)}^2  \le C_1 \big(\norm{g}{L_2(\mu)}^2 + \norm{\nx^h g}{L_2(\mu)}^2\big)- 8C_0^2\int_{\Gamma} \av{g}^T\jump{g}_x \nabla V\ud \nu ,
 	\end{equation}
 	with $C_1:= \max\{2C_0^2(1+2\sqrt{d}+2dC_0^2), 16C_0^2\}$, for $C_0$ as in \eqref{eq:Vgrowth}. In the special case of a Gaussian confining potential $V(x)=|x|^2/2$, we have $C_1=16$.
 \end{lemma}
 \begin{proof} We first show that $	\norm{\mathcal{H}( V) g}{L_2(\mu)}^2\le 2C_0^2\norm{g}{L_2(\mu)}^2+2C_0^2\norm{|\nabla V| g}{L_2(\mu)}^2$.
 	Observing that $-\nabla \ex^{-V}=\nabla V  \ex^{-V}$, we have
 	\[
 	\begin{aligned}
 		\norm{|\nabla V| g}{L_2(\mu)}^2=
 		&
 		\int (\nabla V\ex^{-E})\cdot( \nabla V|g|^2) \ud v\ud x\\
 		= & \int \Delta  V |g|^2\ud \mu+2 \int \nabla V\cdot (\nx^h  g)^T g\ud \mu
 		-\sum_{T\in\mathcal{T}}\int_{\partial T} (\nabla V\cdot n_x)|g|^2\ud \nu
 	\end{aligned}
 	\]
 	Noting the identity
 	\[
 	-\sum_{T\in\mathcal{T}}\int_{\partial T} (\nabla V\cdot n_x)|g|^2\ud \nu = 
 	-\sum_{T\in\mathcal{T}}\int_{\partial T} g^T(g\otimes n_x) \nabla V\ud \nu = -2\int_{\Gamma} \av{g}^T\jump{g}_x \nabla V\ud \nu,
 	\]
 we infer
 	\[
 	\begin{aligned}
 		\norm{|\nabla V| g}{L_2(\mu)}^2
 		\le &\ 	 \sqrt{d}C_0 \norm{g}{L_2(\mu)}^2 +\norm{|\nabla V|g}{L_2(\mu)}\big( \sqrt{d}C_0\norm{ g}{L_2(\mu)} +2\norm{\nx^h  g}{L_2(\mu)}\big)\\
 		&-2\int_{\Gamma} \av{g}^T\jump{g}_x \nabla V\ud \nu,
 	\end{aligned}
 	\]
 	using \eqref{eq:Vgrowth}, along with the Cauchy-Schwarz inequality, respectively, since  $|\Delta V|^2 \leq d|\mathcal{H}( V)|^2 \leq d C_0^2(1+|\nabla V|)^2 $, by \eqref{eq:Vgrowth},
 	with $\nx g$ denoting the Jacobian of $g$ with respect to $x$.
 	Applying Young's inequality on  $\norm{|\nabla V|g}{L_2(\mu)}\big( \sqrt{d} C_0\norm{ g}{L_2(\mu)} +2\norm{\nx^h  g}{L_2(\mu)}\big)$, gives
 	\begin{equation}\label{magic_growth}
 			\norm{|\nabla V| g}{L_2(\mu)}^2  \le (2\sqrt{d}C_0 + 2dC_0^2) \norm{g}{L_2(\mu)}^2 +8 \norm{\nx^h g}{L_2(\mu)}^2-4\int_{\Gamma} \av{g}^T\jump{g}_x \nabla V\ud \nu.
 	\end{equation}
 	Finally, we resort once more into \eqref{eq:Vgrowth}, which implies $	\norm{\mathcal{H}( V) g}{L_2(\mu)}^2\le 2C_0^2\norm{g}{L_2(\mu)}^2+2C_0^2\norm{|\nabla V| g}{L_2(\mu)}^2$ and combine with \eqref{magic_growth} to deduce the result.
 \end{proof}

\begin{remark}
The bound \eqref{eq:villani_control_discrete} is required for $g=\nv U$ below. If the mesh is of $(v,x)$-tensor-product type, then $\jump{\nv U}_x=0$, and the last term on the right-hand side of \eqref{eq:villani_control_discrete} vanishes.
\end{remark}

\begin{lemma}\label{lem:hypoco_discrete}  Assume that \eqref{eq:Vgrowth} and \eqref{growth_V_second} hold. Define
	$\sigma:\Gamma\to\mathbb{R} $ by
		$$
		\sigma|_F:=  2 (c_{hc,2}^{-1}+c_{hc,3}^{-1}) \max\big\{m_{T}C_{\rm inv}(T,p,k),m_{T'}C_{\rm inv}(T',p,k)\big\},
		$$
 with $C_{\rm inv}(T,p,k)$ as in Theorem \ref{lem:gen_in_est_final}, for each element face $F\subset\partial T\cap \partial T'$, shared by two adjacent elements $T,T'\in\mathcal{T}$, and $m_{T}$ denoting the number of $(2d-1)$-dimensional faces of $T$. Select $\mathcal{A}$ satisfying $\alpha\gamma-\beta^2\ge 0$ as in the proof of Lemma \ref{lem:hypoco} and set $\tau:=25\alpha^3C_0/(144C_1)$ for $(v,x)$-unstructured meshes, or $\tau:=0$ for  $(v,x)$-tensor-product ones.Then, for all $W\in V_h$, the following estimate holds:
	\[
	\begin{aligned}
		a_h(W,W)+s_h(W,W)\!\ge\!
		&\ \tilde{c}_{hc,1}\big( \norm{\nv W}{L_2(\mu)}^2 \!+\!\beta \norm{\nx W}{L_2(\mu)}^2\big)
		+c_{hc,2}\frac{\alpha}{2} \norm{\nv^h\nv^T W}{L_2(\mu)}^2 \\
		&	+ c_{hc,3}\frac{\gamma}{2}\norm{\nx^h\nv^T W}{L_2(\mu)}^2
		+ \frac{1}{2}\sum_{j=1}^d\int_{\Gamma}\sigma |\sqrt{\mathcal{C}}\jump{\nvx W}_{v_j}|^2\ud \nu,
	\end{aligned}
	\]
	holds for $\tilde{c}_{hc,1}$, defined like $c_{hc,1}$ from Lemma \ref{lem:hypoco}, by replacing $c_1$ with $C_1$ from Lemma \ref{lem:villani_control_discrete}, and $c_{hc,i}$, $i=2,3$ as in Lemma \ref{lem:hypoco}.
\end{lemma}
\begin{proof}
We consider each term of $a_h(W,W)$ separately.
As before, we have
$
((\mathcal{V}\cdot \nvx)  W,W)_{L_2(\mu)}=0
$,
and $\mathcal{A}$ is a constant matrix.
The (lack of higher) regularity of functions in $V_h\subset H^1(\mu)$ requires partially different treatment compared to the proof of Lemma \ref{lem:hypoco} for the terms involving higher order derivatives. This is because additional terms now appear on the mesh skeleton $\Gamma$, which will eventually be controlled by the stablisation term $s_h$. Indeed, we have
\[
\begin{aligned}
 (( \mathcal{V}\cdot \nvx^h )\nvx  W, \mathcal{A}\nvx  W)_{L_2(\mu)}
=&\	\sum_{T\in\mathcal{T}}\int_{\partial T}\frac{\mathcal{V}\cdot \mbf{n}}{2}|\sqrt{\mathcal{A}}\nvx  W|^2\ud \nu\\
= &\	\int_{\Gamma}\av{\nvx  W}^T\mathcal{A}\jump{\nvx  W}\mathcal{V}\ud \nu.
\end{aligned}
\]
In the following we will use Young's inequality exactly as in the proof of Lemma \ref{lem:hypoco} and we use the choices  $\beta=\alpha^2/2$ and $\gamma=2\alpha^3/3$.
The corresponding bounds to \eqref{eq:HV_one} and \eqref{eq:HV_two}, using \eqref{eq:villani_control_discrete} (instead of \eqref{eq:villani_control},)
produce the additional term
\[
+\frac{25\alpha^3C_0}{144C_1}\int_{\Gamma} \av{\nv  W}^T\jump{\nv  W}_x \nabla V\ud \nu,
\]
 whenever $\jump{\nv  W}_x \neq 0$.
 From the above, we conclude
\begin{equation}\label{eq:hypoco_discrete_one}
	\begin{aligned}
		a_h(W,W)\ge
		&\ c_{hc,1}\big( \norm{\nv W}{L_2(\mu)}^2 +\beta \norm{\nx W}{L_2(\mu)}^2\big)	+c_{hc,2}\alpha \norm{\nv^h\nv^T W}{L_2(\mu)}^2 \\
		&\ + c_{hc,3}\gamma\norm{\nx^h\nv^T W}{L_2(\mu)}^2+	\int_{\Gamma}\av{\nvx  W}^T\mathcal{A}\jump{\nvx  W}\mathcal{V}\ud \nu\\
		&\ + \frac{25\alpha^3C_0}{144C_1}\int_{\Gamma} \av{\nv  W}^T\jump{\nv  W}_x \nabla V\ud \nu,
	\end{aligned}
\end{equation}
with $\tilde{c}_{hc,1}=\min\big\{\frac{2}{5}+\frac{4-5C_1}{10}\alpha-\frac{3}{20}\alpha^2-\frac{25}{72}\alpha^3, \frac{1}{6}-\alpha C_1   \big\}$, $c_{hc,2}=1/6$, and $c_{hc,3}=7/240$. Select $\alpha>0$ small enough, so that $\tilde{c}_{hc,1}>0$.

Set $\tau= 25\alpha^3C_0/(144C_1)$, whenever $\jump{\nv  W}_x \neq 0$, or $\tau=0$ otherwise. From the definition of the stabilisation, we have
\[
\begin{aligned}
s_h(W,W)=&\ -	\int_{\Gamma}\av{\nvx  W}^T\mathcal{A}\jump{\nvx  W}\mathcal{V}\ud \nu-\frac{25\alpha^3C_0}{144C_1}\int_{\Gamma} \av{\nv  W}^T\jump{\nv  W}_x \nabla V\ud \nu\\
&-2\sum_{j=1}^d\int_{\Gamma} \av{\nvx^T W_{v_j}}\mathcal{A}\jump{\nvx  W}_{v_j}\ud \nu
+\sum_{j=1}^d\int_{\Gamma}\sigma|\sqrt{\mathcal{C}}\jump{\nvx  W}_{v_j}|^2\ud \nu,
\end{aligned}
\]
noting that, crucially, the first two terms on the right-hand side cancel out with the last two terms on the right-hand side of \eqref{eq:hypoco_discrete_one}.

So it remains to bound the last two terms appearing in the stabilisation. To that end,
recall the definition of matrix $\mathcal{C}$ and observe the identity
\[
\int_{\Gamma} \av{\nvx W_{v_j}}^T\mathcal{A}\jump{\nvx  W}_{v_j}\ud \nu
=
\int_{\Gamma} \av{{\rm diag}(\sqrt{\alpha} I,\sqrt{\gamma} I)\nvx W_{v_j}}^T\sqrt{\mathcal{C}}\jump{\nvx  W}_{v_j}\ud \nu,
\]
for each $j=1,\dots,d$. The right-hand side of the last identity can be further bounded, using the trace inverse estimate from  Theorem \ref{lem:gen_in_est_final}, by:
\[
\begin{aligned}
&\Big(\int_{\Gamma} \sigma^{-1}|\av{{\rm diag}(\sqrt{\alpha} I,\sqrt{\gamma} I)\nvx W_{v_j}}^T|^2\ud \nu\Big)^{1/2} \Big(\int_{\Gamma}\sigma |\sqrt{\mathcal{C}}\jump{\nvx  W}_{v_j}|^2\ud \nu\Big)^{1/2}\\
\le& \Big(\frac{c_{hc,2}}{2}\alpha \norm{\nv^h W_{v_j}}{L_2(\mu)}^2  + \frac{c_{hc,3}}{2}\gamma\norm{\nx^h W_{v_j}}{L_2(\mu)}^2\Big)^{1/2} \Big(\int_{\Gamma}\sigma |\sqrt{\mathcal{C}}\jump{\nvx  W}_{v_j}|^2\ud \nu\Big)^{1/2},
\end{aligned}
\]
due to the choice of $\sigma$. Adding now all contributions for $j=1,\dots,d$, we conclude
\[
\begin{aligned}
\sum_{j=1}^d\int_{\Gamma} \av{\nvx W_{v_j}}^T\mathcal{A}\jump{\nvx  W}_{v_j}\ud \nu
\le&\   c_{hc,2}\frac{\alpha }{4}\norm{\nv^h\nv^T W}{L_2(\mu)}^2  +c_{hc,3} \frac{\gamma}{4}\norm{\nx^h\nv^T W}{L_2(\mu)}^2
\\
&+\frac{1}{2} \sum_{j=1}^d\int_{\Gamma}\sigma |\sqrt{\mathcal{C}}\jump{\nvx W}_{v_j}|^2\ud \nu.
\end{aligned}
\]
Combining the above already gives the result.
\end{proof}

	For all $w$ in $H^1(\mu)$, we define the `$\mathcal{A}$-norm' by
	\begin{equation}\label{def:A-norm}
		\Anorm{w}:=\sqrt{\tltwo{w}{w}} = \big(\norm{w}{L_2(\mu)}^2
		+ \norm{\sqrt{\mathcal{A}} \nvx  w}{L_2(\mu)}^2\big)^{1/2}.
	\end{equation}
	Let $\Pi$ denote the $\mathcal{A}$-orthogonal projection $\dip{\Pi u_0}{W} = \dip{u_0}{W}$ for all $W\in V_h$.

 We will also use the spectral gap/Poincar\'e inequality:
 \begin{equation}\label{PF}
  	\norm{ w-\bar{w}}{L_2(\mu)}^2\le C_{PF}	\norm{\nvx w}{L_2(\mu)}^2,
 \end{equation}
  for all $w\in H^1(\mu)$, with $\bar{w}:=\int w\ud \mu$. As discussed in \cite{villani}, (see, in particular, \cite[(7.4)]{villani} and \cite[Theorem A.2]{villani}),) \eqref{PF} holds if $V$ satisfies \eqref{eq:Vgrowth} and if $|\nabla V(x)|\to \infty$ as $|x|\to\infty$, which is the case here. Note that \eqref{PF} holds with $C_{PF}=1$ for Gaussian $V$.

We now prove the decay of the numerical method as $t\to \infty$ in this norm.
\begin{theorem}[Numerical hypocoercivity] \label{theorem: semi discrete}
	Set  $U_0 : = \Pi u_0$. With the assumptions of Lemma \ref{lem:hypoco_discrete}, we have
	\begin{equation}\label{eq:num_hypo}
\Anorm{U(t_f)-\bar{u}}^2
	\le
	\ex^{-\kappa t_f}
\Anorm{U_0-\bar{u}}^2,
	\end{equation}
	with $\kappa = 2c_{hc,1}\min\{  \lambda_{\rm min}C_{PF}^{-1},1\}>0$, and $ \lambda_{\rm min} := \min\{ 1-2 \alpha ,\frac{\alpha^2}{2}(1 -\frac{11\alpha}{6}) \}$.
	\end{theorem}
\begin{proof}
Let $U\in V_h$ be the solution of \eqref{eq:method_semi_discrete}. From Lemma \ref{lem:hypoco_discrete}, we have
\[
	a_h(U,U)+s_h(U,U)\ge c_{hc,1}\big( \norm{\nv U}{L_2(\mu)}^2 +\beta \norm{\nx U}{L_2(\mu)}^2\big) = c_{hc,1}\norm{\sqrt{\mathcal{D}} \nvx  U}{L_2(\mu)}^2,
\]
for the matrix $\mathcal{D}:={\rm diag}( I, \beta I) $. We consider the decomposition
$$
\norm{\sqrt{\mathcal{D}} \nvx  U}{L_2(\mu)}^2 = \norm{\sqrt{\mathcal{D-A}} \nvx  U}{L_2(\mu)}^2 + \norm{\sqrt{\mathcal{A}} \nvx  U}{L_2(\mu)}^2.
$$
Now,
\[
\begin{aligned}
	 \norm{\sqrt{\mathcal{D-A}} \nvx  U}{L_2(\mu)}^2 \!=& (1-\alpha)\norm{\nv U}{L_2(\mu)}^2\!
	 -\!2\beta (\nv U,\nx U)_{L_2(\mu)}
	 \! +\! (\beta-\gamma)\norm{\nx U}{L_2(\mu)}^2\\
	 \geq & \	(1-2 \alpha)\norm{\nv U}{L_2(\mu)}^2
		  + (\beta-\gamma -\beta^2/\alpha)\norm{\nx U}{L_2(\mu)}^2 \\
		  \ge&\  \lambda_{\rm min} \norm{\nvx  U}{L_2(\mu)}^2,
\end{aligned}
\]
by Cauchy-Schwarz and Young's inequalities, and substituting $\beta=\alpha^2/2$, $\gamma=2\alpha^3/3$,
with $ \lambda_{\rm min}= \min\{ 1-2 \alpha ,\frac{\alpha^2}{2}(1 -\frac{11\alpha}{6}) \}$. Thus, using \eqref{PF}, and recalling \eqref{eq:mass_cons}, gives
\begin{align}\label{Poncare inequality}
	a_h(U,U)+s_h(U,U)\ge&\ c_{hc,1}\big( \lambda_{\rm min}\norm{\nvx  U}{L_2(\mu)}^2 +  \norm{\sqrt{\mathcal{A}} \nvx  U}{L_2(\mu)}^2\big) \nonumber \\
\ge	&	\ c_{hc}\big( \lambda_{\rm min}C_{PF}^{-1} \norm{ U - \bar{u}}{L_2(\mu)}^2 +  \norm{\sqrt{\mathcal{A}} \nvx  U}{L_2(\mu)}^2\big)
\ge		 \frac{\kappa}{2} \Anorm{U-\bar{u}}^2.
\end{align}
Recalling Remark \ref{rem:discrete_cons},  for $\eta\in V_h$, we observe the identity $a_h(U,\eta)+s_h(U,\eta) = a_h(U-\bar{u},\eta)+s_h(U-\bar{u},\eta)$. Thus, setting $\eta = U-\bar{u}$ for brevity, \eqref{eq:method_semi_discrete} implies
\begin{align*}
 0 =& \frac{1}{2}\frac{d}{dt}\Anorm{ \eta}^2+a_h(\eta,\eta)+s_h(\eta,\eta)
\geq
\frac{1}{2}\frac{d}{dt}\Anorm{\eta}^2
+ \frac{\kappa}{2} \Anorm{\eta}^2.
\end{align*}
The result follows by Gr\"onwall's Lemma,
and the identity $U_0 = \Pi u_0=\bar{u}$.
\end{proof}

\section{A fully discrete scheme} \label{sec: A fully discrete scheme}
We now further equip the above with a variable order ($hp$-version) discontinuous Galerkin (dG) time-stepping for the discretization of the temporal derivatives, thereby arriving at a fully discrete scheme. $hp$-dG timestepping is preferred to effectively capture possible singularities in the neighbourhood of the initial time \cite{schotzau_schwab}.

To that end, we consider a subdivision $\{I_n \}_{n=1}^{N_t}$ of the time interval $(0,t_f]$, subordinate to the time nodes $0=: t_0 <t_1<\dots <t_{N_t}:=t_f$, with  $I_n = (t_{n-1},t_n]$.
  Let  also $\tmesh:=t_{n} - t_{n-1}$. To each $I_n$ we associate the local polynomial degree $q_n\in \mathbb{N}\cup\{0\}$ of the temporal basis functions, and we collect them in the vector $\textbf{q}=(q_1,\dots,q_N)$. On each $I_n$, we define the local \emph{space-time finite element space} by
\[
\stnfes{n}:=\{ u\in L_2(I_n;H^1(\mu)): u|_{I_n}\in \mathcal{P}_{q_n}(I_n)\times V_h, \ n=1,\dots,N_t \},
\]
with $\mathcal{P}_{q}(I)$ denoting the space of univariate polynomials of degree at most $q$ on the interval $I$. The global  space-time finite element space is then given by $\stfes = \bigoplus_{n=1}^{N_t} \stnfes{n}$. Functions in $\stfes$ are allowed to be discontinuous at the time nodes $\{t_n\}_{n=1}^{N_t-1}$. Defining
$
u_n^\pm := \lim_{s\rightarrow 0^+} u(t_n\pm s)$, $  n=0,1,\dots, N_t$,
the \emph{time-jump} across $t_{n}$ is given by
$\ujump{u}_{n}  := u_{n}^+-u_{n}^-$, for $n=1,\dots, N_t$ and $\ujump{u}_{0}  := u_{0}^+$. The fully discrete scheme then reads: find $U\in \stfes$, such that
\begin{equation}\label{eq:method_full_discrete}
B(U,W)=\tltwo{u_{0}}{W^+_{0}} , \quad\text{ for all }W\in \stfes,
\end{equation}
with  $u_0 = u(0,v,x)$, and $B:\stfes \times \stfes \to \mathbb{R}$ given by
\[
B(U,W) := \sum_{n=1}^{N_t}\! \bigg(\int_{I_n} \Big( \tltwo{U_t}{W}  + a_h(U,W) +s_h(U,W)\Big)  \ud  t +
\tltwo{\ujump{U}_{n-1}}{W^+_{n-1}}\bigg).
\]

Of course, in practice, the approximation is computed recursively: for $n=1,\dots,N_t$, $U|_{I_n}\in \stnfes{n}$ is computed by
\begin{equation}\label{local-bilinear}
 \int_{I_n}  \Big( \tltwo{U_t}{W}  + a_h(U,W) +s_h(U,W)\Big)  \ud  t
+
 \tltwo{U^+_{n-1}}{W^+_{n-1}}
 =  \tltwo{U^-_{n-1}}{W^+_{n-1}},
\end{equation}
for all $W\in \stnfes{n}$,
with $U_{n-1}^-$  the initial datum at time step $I_n$; for $n=1$  set $U_0^- = u_0$.

We now present an analogous result to Theorem \ref{theorem: semi discrete} for the fully discrete scheme \eqref{eq:method_full_discrete}. The key underlying mathematical challenge to be addressed is that  \eqref{eq:method_full_discrete} is defined in integral form and, thus, Gr\"onwall's Lemma with negative exponents does not apply.
The lack of point-wise error control in the $t$ variable of $U$ solving \eqref{eq:method_full_discrete} is counteracted by an inverse inequality.

\begin{theorem}[Numerical hypocoercivity via inverse inequality]\label{Numerical hypocoercivity via inverse inequality}
	With the assumptions and notation of Lemma \ref{lem:hypoco_discrete}, the solution $U$ to \eqref{eq:method_full_discrete} satisfies
	\begin{equation}\label{eq:num_hypo fully discrete}
	\Anorm{U(t_f)-\bar{u}}^2
	\le
\prod_{n=1}^{N_t}\Big(1+ \frac{\kappa \tau_n}{(q_n +1)^2}\Big)^{-1}
	\Anorm{U_0-\bar{u}}^2.
	\end{equation}
	\end{theorem}
\begin{proof}
Set $\eta=U-\bar{u}$. Since $a_h(U,\eta)+s_h(U,\eta) = a_h(\eta,\eta)+s_h(\eta,\eta)$, \eqref{local-bilinear} implies
\begin{equation*}
\begin{split}
 \int_{I_n}  \Big(  \frac{1}{2} \frac{d}{dt}\Anorm{\eta}^2+ a_h(\eta,\eta) +s_h(\eta,\eta)\Big)  \ud  t
+ \tltwo{\ujump{\eta}_{n-1}  }{\eta^+_{n-1}}
 =  0.
\end{split}
\end{equation*}
Hence,
\begin{equation*}
\frac{1}{2} \Anorm{\eta_n^-}^2+ \frac{1}{2} \Anorm{\eta_{n-1}^+}^2- \tltwo{\eta^-_{n-1}  }{\eta^+_{n-1} }
+ \int_{I_n}  \big( a_h(\eta,\eta) +s_h(\eta,\eta)\big)  \ud  t =0.
\end{equation*}
Since $ \frac{1}{2} \Anorm{\eta_{n-1}^+}^2- \tltwo{\eta^-_{n-1}  }{\eta^+_{n-1} } = \frac{1}{2} \Anorm{\ujump{\eta}_{n-1}}^2- \frac{1}{2} \Anorm{\eta_{n-1}^-}^2$, the above gives
\begin{equation*}
\Anorm{\eta_n^-}^2 +  \Anorm{\ujump{\eta}_{n-1}}^2\\
+ 2 \int_{I_n}  \big( a_h(\eta,\eta) +s_h(\eta,\eta)\big)  \ud  t
= \Anorm{\eta_{n-1}^-}^2.
\end{equation*}
Following the proof in Theorem \ref{theorem: semi discrete}, we deduce
\begin{equation}\label{relation 1}
\begin{split}
 \Anorm{\eta_n^-}^2 +\Anorm{\ujump{\eta}_{n-1}}^2
+  \kappa \int_{I_n}  \Anorm{U-\bar{u}}^2 \ud  t
\leq \Anorm{\eta_{n-1}^-}^2.
\end{split}
\end{equation}
Employing the classical one-dimensional $hp$-version trace inverse inequality from \cite[Theorem 2]{warburton2003constants} in the $t$ variable, we have
$
\tau_n \Anorm{\eta_n^-}^2  \leq (q_n+1)^2 \int_{I_n}  \Anorm{\eta}^2 \ud  t,
$
which upon substitution into \eqref{relation 1} gives
$$
\Big(1+\frac{\kappa\tau_n}{(q_n+1)^2}\Big) \Anorm{\eta_n^-}^2  \leq    \Anorm{\eta_{n-1}^-}^2.
$$
The result follows by repeated application of the last bound for $n=1,\dots, N_t$.
\end{proof}
\begin{remark}[Asymptotic bound]
Assuming constant $\tau_n=\tau$ and $q_n=q$ for all  $n=1,\dots,N_t$, we deduce
\[
	\Anorm{U(t_f)-\bar{u}}^2
\leq
\Big(1+ \frac{\kappa \tau}{(q +1)^2}\Big)^{-t_f/\tau}	\Anorm{U_0-\bar{u}}^2
\rightarrow
e^{- \kappa  t_f/(q +1)^2}	\Anorm{U_0-\bar{u}}^2,
\]
as $\tau\rightarrow0$. Comparing with \eqref{eq:num_hypo}, the exponent coincides with that of the semidiscrete setting for $q=0$. The applied inverse estimate results in a reduction of the exponent for higher order. Nevertheless, the exponential convergence to equilibrium is established for all $q$.
\end{remark}

\section{Numerical examples}\label{sec: Numerical examples}
We consider three numerical experiments to investigate the practical performance of the proposed methods. In the first two examples, we demonstrate the exponential decay of the $\mathcal{A}$-norm for the problem \eqref{eq:num_hypo fully discrete}. In the third example, we study the convergence rate of the error measured in the $\mathcal{A}$-norm.

	Throughout this section, we choose
	$
	E(v,x):=\frac{1}{2}|v|^2 +\frac{1}{2}|x|^2
	$.
	For this choice, we have $C_1 = 16$, and we set $\alpha = 1/192$ throughout. Consequently,
	$
	\tilde{c}_{hc,1} = 1/12$.
In all three examples, we consider the domain $\Omega:=(-5,5)^2$, as the  partitioned into a sequence of uniform triangular `$(v,x)$-unstructured' meshes consisting of $32$, $128$, $512$, $2048$, and $8192$ elements. This is extended to the whole domain notionally by `infinite' elements as per Figure \ref{fig:meshes} (left panel). However,  for this choice of domain we have $\ex^{-E(v,x)}\le 3.73\times 10^{-6}$ on $\partial \Omega$, reaching  $\ex^{-E(v,x)}\approx1.39\times 10^{-11}$ at the corner points $(\pm5, \pm5)$,
and so the contributions of the `infinite' elements are negligible.
The temporal polynomial degree $q$ and the spatial polynomial degree $p$ are chosen to be equal throughout.

\subsection{Example 1. Convergence to equilibrium for smooth initial condition}
For the first experiment, we set $t_f = 40$ and choose $\tau_n = 4\sqrt{2}\,h$, for all $n=1,\dots,N_t$.
The initial condition is chosen as
$
u_0(v,x) = \sin(\pi v)\sin(\pi x)$.

We plot the quantity $\Anorm{U(t_n^-)-\bar{u}}$ at the right endpoints of the time intervals, together with its value at the initial time.
 The corresponding results are shown in Figure~\ref{fig:convergenceToEquilibrium_sin} (left panel): the numerical solution, measured in the $\Anorm{U(t_n^-)-\bar{u}}$-norm (vertical axis) is drawn against $t$ (horizontal axis). We observe exponential convergence to the equilibrium state.
\begin{figure}\centering
	\includegraphics[width = .385\linewidth]{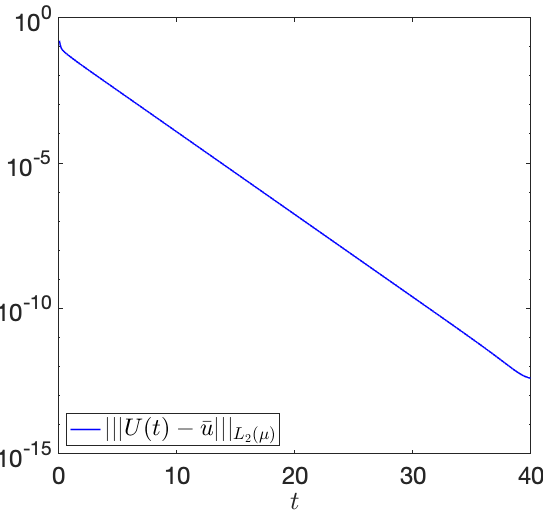}\hspace{1cm}
		\includegraphics[width = .4\linewidth]{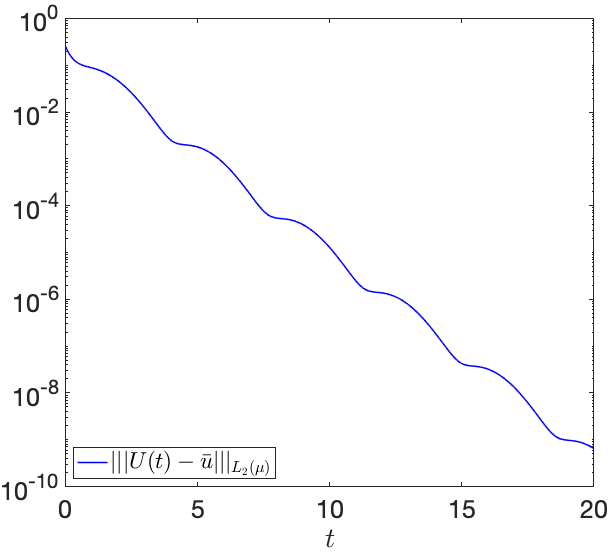}
	\caption{Convergence to equilibrium for $p=2$, $q=2$ with $8192$ triangular elements: Example 1 (left panel) and Example 2 (right panel)}\label{fig:convergenceToEquilibrium_sin}
\end{figure}

\subsection{Convergence to equilibrium for non-smooth initial condition}
For the second experiment, we set $t_f = 20$ and choose  $\tau_n  = 2\sqrt{2}\,h$,  for all $n=1,\dots,N_t$. The initial condition is chosen as
$u_0(v,x) = 1 - \tanh\!\big(\sqrt{v^2+x^2}\big)$. The corresponding results are shown in Figure~\ref{fig:convergenceToEquilibrium_sin} (right panel). We observe that the numerical solution converges exponentially fast to the equilibrium state in the $\Anorm{\cdot}$-norm.

In addition, Figure~\ref{fig:numerical lolution at different time steps} displays the numerical solution at $t=0$, $1$, $5$, and $10$. As time evolves, the solution becomes increasingly uniform and converges towards a constant equilibrium state, in agreement with the exponential decay observed in the $\Anorm{\cdot}$-norm.
\begin{figure}\centering
	\begin{tabular}{cc}
		\includegraphics[width = .45\linewidth]{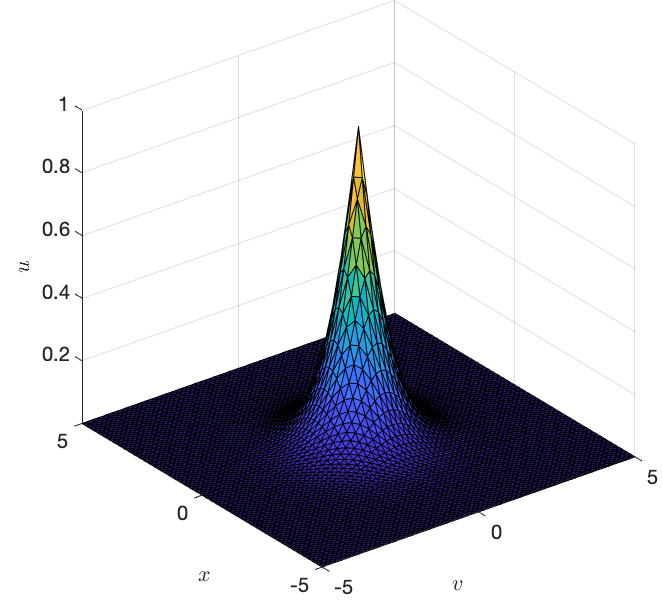}&
		\includegraphics[width = .45\linewidth]{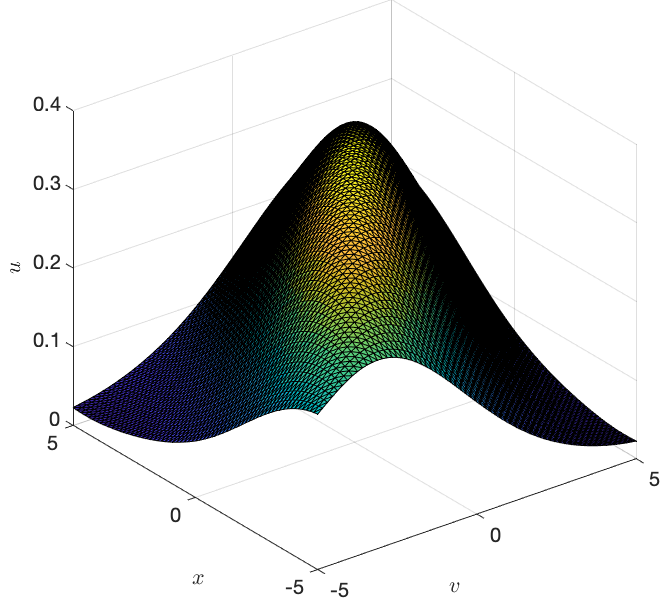}\\
		$t=0$ &
		$t=1$\\
		\includegraphics[width = .45\linewidth]{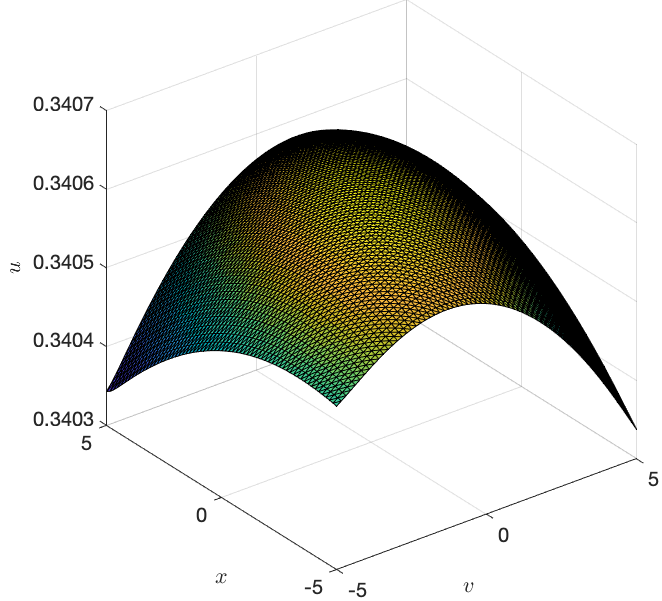}&
		\includegraphics[width = .45\linewidth]{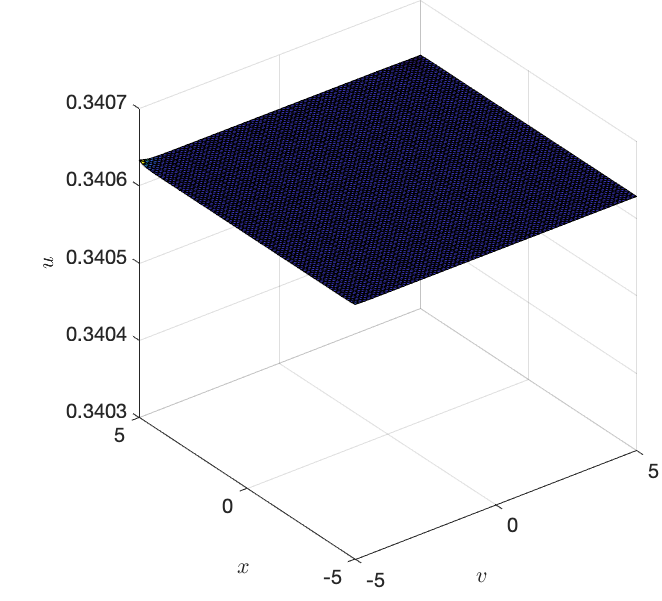}\\
		$t=5$ &
		$t=10$\\
	\end{tabular}
	\caption{Example 2. Numerical solution for $8192$ triangular elements with $p=q=2$.}\label{fig:numerical lolution at different time steps}
\end{figure}

Moreover, Figure~\ref{fig: variable q} presents a numerical investigation of the dependence on $q$ of the exponential decay of the numerical approximation in~\eqref{eq:num_hypo fully discrete}. We computed the solution on a mesh consisting of $8192$ triangular elements, using $p=2$ and $q=0,1,2$. The three exponential decay curves are virtually indistinguishable. This observation suggests that the dependence of the theoretical decay rate on $q$ established in Theorem~\ref{Numerical hypocoercivity via inverse inequality} may pessimistic for $q\ge1$.

\begin{figure}\centering
		\includegraphics[width = .45\linewidth]{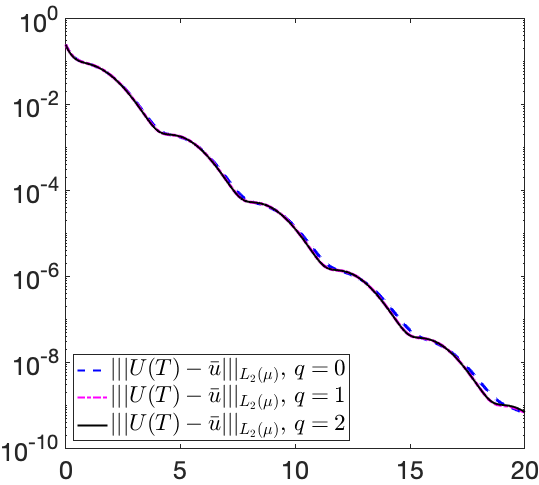}
	\caption{Convergence to equilibrium for $p=2$, $q=0,1,2$ with $8192$ triangular elements Example 2}\label{fig: variable q}
\end{figure}

\subsection{Convergence rate study}
For the third experiment, we study the convergence rate behaviour of the method \eqref{eq:method_full_discrete}. We set $t_f = 1$ and choose $\tau_n = \sqrt{2}\,h/20$, for all $n=1,\dots,N_t$. We measure the error in the $L^\infty(0,t_f;\Anorm{\cdot})$-norm. We consider the manufactured solution
$u(t,v,x) = \sin(\pi v)\sin(\pi x)\sin(\pi t)$.

The results are presented in Figure~\ref{fig:error_smooth}. We observe that the $L^\infty(0,t_f;\Anorm{\cdot})$-norm error converges at the rate $\mathcal{O}(h^{p-1})$ for $p=q=2,3,4$. Since the exact solution is oscillatory in both space and time, the optimal convergence rate is attained only when the time-step and spatial mesh sizes are sufficiently small to resolve these oscillations.
\begin{figure}\centering
	\includegraphics[width = .5\linewidth]{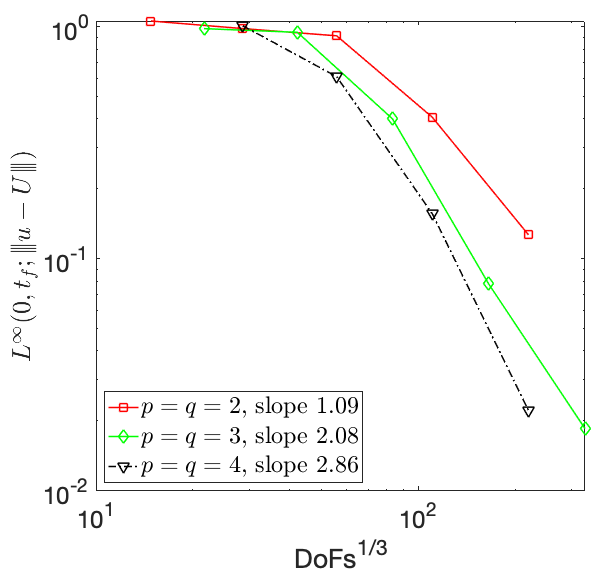}
	\caption{Convergence rate for a sequence of meshes with $p=q=2,3,4$.}\label{fig:error_smooth}
\end{figure}

\section{Concluding remarks}\label{sec: conclusion}
The above class of hypocoercivity-preserving fully discrete Galerkin methods for the (inhomogeneous) kinetic Fokker--Planck (kFP) equations allows for extremely general unstructured mesh scenarios, thereby facilitating mesh and order adaptivity frameworks in the context of kinetic equations. The proof of exponential convergence to the equilibrium state is oblivious the specificities of the meshes used and, in general, of the Galerkin spaces, as long as they produce a globally continuous, element-wise polynomial approximation and contain the global constant function to allow for mass conservation. An interesting direction of further research is to the proof of long-time robust error bounds for the proposed method. This will be considered elsewhere.

	\section*{Acknowledgements}
EHG wishes to acknowledge the financial support of EPSRC (grant number EP/W005840/2).
	\bibliographystyle{siam}
	\bibliography{bib.bib}
\end{document}